\documentclass{rae}
\usepackage{amsmath,amssymb,amsthm,euscript}
\usepackage{hyperref}

\title{TRANSSERIES FOR BEGINNERS}
\author{G.A. Edgar,
 Department of Mathematics, The Ohio State University,
 Columbus, Ohio 43210, USA
 (e-mail: edgar@math.ohio-state.edu)
 }

\theoremstyle{plain}
\newtheorem{pr}{Proposition}
\newtheorem{thm}[pr]{Theorem}

\theoremstyle{remark}
\newtheorem{re}[pr]{Remark}
\newtheorem{de}[pr]{Definition}
\newtheorem{no}[pr]{Notation}
\newtheorem{ex}[pr]{Example}

\newtheorem{problm}[pr]{Problem}

\newtheorem{comq}[pr]{$\bigcirc\;$ Comment}
\numberwithin{pr}{section}

\newcommand{\ec}{\hfill$\bigcirc$\end{comq}} 
\newenvironment{com}{\begin{comq}}{\ec} 
\newenvironment{comnoc}{\begin{comq}}{\end{comq}} 

\newcommand{\Def}[1]{\textbf{\itshape #1}} 

\renewcommand{\phi}{\varphi}
\renewcommand{\epsilon}{\varepsilon}
\renewcommand{\emptyset}{\varnothing}
\newcommand{\takes}{\colon}
\newcommand{\fgt}{\succ}
\newcommand{\fgte}{\succcurlyeq}
\newcommand{\fst}{\prec}
\newcommand{\fste}{\preccurlyeq}

\newcommand{\cto}{\asymp}
\newcommand{\ato}{\sim}

\newcommand{\SET}[2]{ \left\{\, {#1} : {#2} \,\right\} }
\newcommand{\R}{\mathbb R}
\newcommand{\C}{\mathbb C}

\newcommand{\N}{\mathbb N}
\newcommand{\Z}{\mathbb Z}

\newcommand{\G}{\mathfrak G}

\newcommand{\GRID}{\mathfrak J}

\renewcommand{\AA}{\mathfrak A}
\newcommand{\BB}{\mathfrak B}
\newcommand{\MM}{\mathfrak M}
\newcommand{\T}{\mathbb T}

\renewcommand{\P}{\EuScript P}
\renewcommand{\L}{\EuScript L}
\newcommand{\Bor}{\EuScript B}

\newcommand{\BE}{\mathbf E}
\newcommand{\BF}{\mathbf F}
\newcommand{\J}{\mathbf J}
\newcommand{\W}{\mathbf W}
\renewcommand{\o}{\mathrm o}
\renewcommand{\O}{\mathrm O}

\newcommand{\showfraktur}[1]{(\;{#1}\;\mathfrak{#1}\;)}

\newcommand{\muto}{\overset{\ebmu}{\longrightarrow}}
\newcommand{\mto}{\overset{\bm}{\longrightarrow}}

\newcommand{\bk}{\mathbf k}
\newcommand{\bp}{\mathbf p}
\newcommand{\bm}{\mathbf m}
\newcommand{\0}{\mathbf 0}
\newcommand{\fa}{\mathfrak a}
\newcommand{\fb}{\mathfrak b}

\newcommand{\g}{\mathfrak g}
\newcommand{\m}{\mathfrak m}
\newcommand{\n}{\mathfrak n}

\renewcommand{\l}{\mathfrak l}
\newcommand{\A}{\mathbf A}

\newcommand{\Gsmall}{\G^{\mathrm{small}} }

\newcommand{\Msmall}{\MM^{\mathrm{small}} }
\newcommand{\supp}{\operatorname{supp}}
\newcommand{\lsupp}{\operatorname{lsupp}}

\renewcommand{\mag}{\operatorname{mag}}

\newcommand{\dom}{\operatorname{dom}}
\newcommand{\Ei}{\operatorname{Ei}}

\newcommand{\Max}{\operatorname{Max}}
\newcommand{\Min}{\operatorname{Min}}

\newcommand{\SUBSECformalseries}{3B}
\newcommand{\SUBSECgenerators}{3C}
\newcommand{\SUBSECtsinf}{3D}
\newcommand{\SUBSECwithlogs}{3E}

\newcommand{\bmu}{{\boldsymbol{\mu}}}
\newcommand{\ebmu}{{\boldsymbol{\mu}}}
\newcommand{\tbmu}{{\widetilde{\bmu}}}
\newcommand{\tebmu}{{\tilde{\ebmu}}}

\newlength{\uuu}
\newcommand{\lbb}{\begin{picture}(8,8)(-2,2)
	\put(0,0){\line(0,1){9}}
	\put(2,0){\line(0,1){9}}
	\put(0,0){\line(1,0){5}}
	\put(0,9){\line(1,0){5}}
	\end{picture}}
\newcommand{\rbb}{\begin{picture}(7,8)(0,2)
	\put(3,0){\line(0,1){9}}
	\put(5,0){\line(0,1){9}}
	\put(0,0){\line(1,0){5}}
	\put(0,9){\line(1,0){5}}
	\end{picture}}
\newcommand{\lbbb}{\begin{picture}(10,8)(-2,2)
	\put(0,0){\line(0,1){9}}
	\put(2,0){\line(0,1){9}}
	\put(4,0){\line(0,1){9}}
	\put(0,0){\line(1,0){7}}
	\put(0,9){\line(1,0){7}}
	\end{picture}}
\newcommand{\rbbb}{\begin{picture}(9,8)(0,2)
	\put(3,0){\line(0,1){9}}
	\put(5,0){\line(0,1){9}}
	\put(7,0){\line(0,1){9}}
	\put(0,0){\line(1,0){7}}
	\put(0,9){\line(1,0){7}}
	\end{picture}}

\begin{document}
\maketitle

\begin{abstract}
From the simplest point of view, transseries are a new kind of
expansion for real-valued functions.  But transseries constitute much more
than that---they have a very rich (algebraic, combinatorial,
analytic) structure.  The set of transseries is a large
ordered field, extending the real number field, and
endowed with additional operations such as exponential,
logarithm, derivative, integral, composition.
Over the course of the last 20 years or so,
transseries have emerged in several areas of mathematics:
asymptotic analysis, model theory, computer algebra, surreal numbers.
This paper is an exposition for the non-specialist mathematician.

\textit{All a mathematician needs to know
in order to apply transseries.}
\end{abstract}

\tableofcontents

\section*{Introduction}
\addcontentsline{toc}{section}{Introduction}%
Although transseries expansions are prominent in certain areas
of mathematics, they are not well known to mathematicians
in general.  Here, I try to bring these beautiful
mathematical structures to the attention of non-specialists.
This paper complements the already-existing survey articles
such as \cite{asch, ressayre}, or monographs
\cite{hoeven, kuhlmann, schmeling}.

Transseries come in various flavors.   Here I focus on one particular
variant---the \Def{real, grid-based} transseries---since they are the ones
which are most amenable to explicit computations, and
transseries representing real-valued functions
naturally arising in analysis (e.g., as solutions to algebraic
differential equations) are usually of this type.
Once familiar with one variant, it should be relatively easy
to work with another.

The major part of this paper (Section~\ref{startformal})
presents a formal construction of the
differential field of real, grid-based transseries. 
Section~\ref{excomp} illustrates
its use in practice through examples: transseries expansions for
functional inverses, for anti-derivatives, for solutions
of differential equations, etc.   The development is entirely
formal; the analytic aspects and origins of the subject
(computer algebra limit algorithms,
\'Ecalle's generalization of Borel summation, Hardy fields, etc.)
are omitted---a survey of that
aspect of the subject would warrant a
separate paper.  This restriction allows for a
self-contained exposition, suited for mathematicians regardless
of their specialties.

There are several constructions of the various fields of
transseries already in the literature, smoothing out and filling
in details in the original papers; for example:
van den Dries--\-Macintyre--Marker \cite{DMM},
van der Hoeven \cite{hoeven}, and Costin \cite{costintop}.
These all require a certain technical apparatus,
despite the simplicity of the basic construction.
Here I try to avoid such requirements and assume only
a minimum of background knowledge.

Sections~\ref{salespitch} and \ref{what}, which are intended to lure the
reader into the transseries world, give examples
of natural computations which can be made precise in this framework.
Section~\ref{startformal} deals with the rigorous construction of grid-based
transseries.  It is intended as: \textit{All a mathematician needs to know
in order to apply transseries.}
Section~\ref{excomp} contains worked-out examples,
partly computed with the aid of computer algebra software.
Section~\ref{addremarks} gives suggestions for further reading.
(This introduction is taken mostly from an anonymous referee's report
for an earlier draft of the paper.  That referee understood
what this paper is about better than I did myself!)

Review of ``Fraktur'' style letters:
\begin{equation*}
\showfraktur{A}
\showfraktur{B}
\showfraktur{G}
\showfraktur{J}
\showfraktur{M}
\showfraktur{a}
\showfraktur{b}
\showfraktur{g}
\showfraktur{l}
\showfraktur{m}
\showfraktur{n}
\end{equation*}

\section{Sales Pitch}\label{salespitch}
One day long ago, I wrote Stirling's
formula like this:
\begin{equation*}
	\log \Gamma(x) = \left(x - \frac{1}{2}\right)\log x - x +
	\frac{\log(2\pi)}{2} + \sum_{n=1}^\infty
	\frac{B_{2n}}{2n(2n-1)x^{2n-1}} ,
\end{equation*}
where the $B_{2n}$ are the Bernoulli numbers.
But my teacher gently told me that the series diverges
for every $x$.  What a disappointment!

Leonhard Euler \cite[p.~220]{euler} (the master of us all \cite{dunham}) wrote:
\begin{equation*}
	\sum_{j=0}^\infty \frac{(-1)^jj!}{x^{j+1}} = -e^{x}\Ei(-x),
\end{equation*}
where the exponential integral function is defined
by
\begin{equation*}
	\Ei(-x) := \int_{-\infty}^{-x} \frac{e^t}{t}\,dt .
\end{equation*}
But later mathematicians sneered at this,
saying that the series diverges wildly.

To study a sequence $a_j$, it is sometimes useful to consider
the ``generating function''
\begin{equation*}
	\sum_{j=0}^\infty a_j z^j = \sum_{j=0}^\infty \frac{a_j}{x^j} .
\end{equation*}
(The change of variables $z=1/x$ was made so that we can
consider not $z$ near zero but $x$ near infinity, as we will always
do here.)  In fact,
it is quite useful to consider such a series ``formally''
even if the series diverges \cite{wilf}.
The generating function for the sequence $2^j$ is of course
\begin{equation*}
	\sum_{j=0}^\infty \frac{2^j}{x^j} = \frac{1}{1-2/x} .
\end{equation*}
But who among you has not secretly substituted $x=1$
to get
\begin{equation*}
	\sum_{j=0}^\infty 2^j = -1
\end{equation*}
and wondered at it?

To study asymptotic behavior of functions,
G. H. Hardy promoted the class $L$ of ``orders of infinity'': all
functions (near $\infty$) obtained starting with constants and $x$,
then applying the field operations,
$\exp$, and $\log$ repeatedly in any order.
Function $x e^x$ is a valid member of that class.
Liouville had shown that its inverse function isn't.
What cruel classification
would admit a function but not its inverse?

Undergraduate courses in ordinary differential equations
tell us how to solve a linear differential equation
with analytic coefficients in terms of power series---at least
at ordinary points, and at regular singular points.  But
power series solutions do not work at irregular singular points.
Is it hopeless to understand solutions near these points?

Solving linear differential equations with
constant coefficients can be approached by factoring
of operators.  Take, for example, $3Y'' - 5Y' -2Y = 3x$.
Writing $\partial$ for the derivative operator and
$I$ for the identity operator, this
can be written $L[Y] = 3x$, where $L = 3\partial^2-5\partial-2I$.
Then factor this polynomial,
$L = 3(\partial-2I)(\partial+(1/3)I)$ and solve
$L[Y]=3x$ with two successive integrations:
First write $Y_1 = (\partial+(1/3)I)Y$.  Then solve
$\partial Y_1 - 2 Y_1 = x$ to get $Y_1 = A e^{2x}-1/4-x/2$.
Then solve $\partial Y + (1/3)Y = A e^{2x}-1/4-x/2$ to get
$Y=(3A/7)e^{2x}+Be^{-x/3}+15/4-3x/2$.  Wouldn't it
be grand if this could be done for linear differential
equations with variable coefficients?  But
we cannot solve the differential equation
$Y'' + xY' + Y = 0$ by factoring
$\partial^2 + x\partial + I =
(\partial - \alpha(x)I) (\partial - \beta(x)I)$,
where $\alpha(x)$ and $\beta(x)$ are polynomials;
or rational functions; or elementary functions.
But what if we could factor with some new, improved, simple,
versatile class of functions?

\medskip
Well, brothers and sisters, I am here today to tell you:
If you love these formulas, you need no longer hide
in the shadows!  The answer to all of these woes is here.

\medskip\emph{Transseries}\medskip

The differential field of transseries was
discovered [or, some would say, invented]
independently in various parts of
mathematics: asymptotic analysis,
model theory, computer algebra, surreal numbers.
Some feel it was surprisingly
recent for something so natural.
Roots of the subject go back to
\'Ecalle \cite{ecalledulac} and Il$'$yashenko
\cite {ilyashenko} working in asymptotic analysis;
Dahn and G\"oring \cite{dahn, dahng} working in model theory;
Geddes \& Gonnet \cite{geddes} working in computer algebra;
Kruskal working in surreal numbers
(unpublished: see the Epilog in the Second Edition of \textit{On
Numbers and Games} \cite{conway}).
They arrived at eerily similar mathematical structures, although
they did not have all the same features.
It is \'Ecalle who recognized the power of these objects, coined the term,
developed them systematically and in their own right, found ``the'' way to
associate functions to them.
[I~am not tracing the history here.
Precursors---in addition to G. H. Hardy,
Levi-Cevita \cite{levi}, du Bois-Reymond \cite{dubois},
even Euler---include
Lightstone \& Robinson \cite{light},
Salvy \& Shackell \cite{salvy},
Rosenlicht \cite{rosenlicht}, and
Boshernitzan \cite{boshernitzan}.
This listing is far from complete:  Additional historical
remarks are in \cite{hoeven, ressayre, shackell}.]

I hope this paper will show that knowledge of model
theory or asymptotic analysis or computer algebra
or surreal numbers is not required in order
to understand this new, beautiful, complex object.

In this paper, we consider only series used for $x \to +\infty$.
Limits at other locations, and from other directions, are
related to this by a change of variable.  For example,
to consider $z \to 1$ from the left, write $z=1-1/x$ or
$x=1/(1-z)$.

\section{What Is a Transseries?}\label{what}
There is an ordered group $\G$ of \Def{transmonomials} and
a differential field $\T$ of \Def{transseries}.  But $\G$ and $\T$
are each defined in terms of the other, in the way
logicians like to do.  There is even
some spiffy notation (taken from \cite{hoeven}):
$\T = \R\lbb\G\rbb = \R\lbbb x \rbbb$.
The definition is carried out
formally in Section~\ref{startformal}.
But for now let's see informally what they look like.
[This is ``informal'' since, for example,
some terms are used before they are defined, so that the whole
thing is circular.]

\begin{enumerate}
\item[(a)] A log-free \Def{transmonomial}
has the form $x^b e^L$, where $b$ is real
and $L$ is a purely large log-free transseries;
``$x$'' and ``$e$'' are just symbols.  Examples:
\begin{equation*}
	x^{-1}, \qquad
	x^{\pi}e^{x^{\sqrt{2}}-3x}, \qquad
	e^{\sum_{j=0}^\infty x^{-j}e^{x}}
\end{equation*}
Use
$x^{b_1}e^{L_1} \cdot x^{b_2}e^{L_2} = x^{b_1+b_2} e^{L_1+L_2}$
for the group operation ``multiplication''
and group identity $x^0 e^0 = 1$.
The ordering $\fgt$ (read ``far larger than'',
sometimes written $\gg$ instead)
is defined for $\G$ lexicographically:
$x^{b_1} e^{L_1} \fgt x^{b_2} e^{L_2}$ iff
$L_1 > L_2$ or \{$L_1 = L_2$ and $b_1 > b_2$\}.
Examples:
\begin{equation*}
	e^{\sum_{j=0}^\infty x^{-j}e^{x}} \fgt e^x
	\fgt x^{-3} e^x
	\fgt x^\pi \fgt x^{-1}
	\fgt x^{-5} \fgt x^{2008} e^{-x}
\end{equation*}
\item[(b)] A log-free \Def{transseries}
is a (possibly infinite) formal sum
$T = \sum_j c_j \g_j$, where the coefficients
$c_j$ are nonzero reals and the $\g_j$ are log-free transmonomials.
``Formal'' means that we want to contemplate the sum
as-is, not try to assign a ``value'' to it.
The sum could even be transfinite (indexed
by an ordinal), but for each
term $c_j \g_j$, the monomial $\g_j$
is far smaller than all previous terms.
Example:
\begin{equation*}
	-4 e^{\sum_{j=0}^\infty x^{-j}e^{x}}
	+ \sum_{j=0}^\infty x^{-j}e^{x}
	- 17 + \pi x^{-1}
\end{equation*}
Transseries are added termwise (even series of transseries, but
each monomial should occur
only a finite number of times, so we can collect them).
Transseries are multiplied in the way suggested
by the notation---``multiply it out''---but again we have
to make sure that each monomial occurs
in the product only a finite number of times.
The transseries $T=\sum c_j \g_j$ is \Def{purely large} iff $\g_j \fgt 1$ for
all terms $c_j \g_j$; and $T$ is \Def{small} iff
$g_j \fst 1$ for all terms $c_j \g_j$.
A nonzero transseries $T=\sum c_j \g_j$
has a \Def{dominant term} $c_0 \g_0$ with
$\g_0 \fgt \g_j$ for all other terms $c_j \g_j$.
If $c_0 > 0$ we say $T > 0$.
An ordering $>$ is then defined by: $S > T$ iff $S-T > 0$.

We consider only transmonomials and transseries of ``finite exponential
height''---so, for example, these are not allowed:
\begin{equation*}
	e^{\displaystyle e^{\scriptstyle e^{ e^{.^{.^.}}+x}+x}+x},
	\qquad
	x^{-1}+e^{-x}+e^{-e^{x}}+e^{-e^{e^x}}+\cdots .
\end{equation*}	

\item[(c)] \Def{Differentiation} is defined as in elementary
calculus:
\begin{equation*}
	\big(x^b e^L\big)' = b x^{b-1} e^L + x^b L' e^L,\qquad
	\left(\sum c_j \g_j\right)' = \sum c_j \g'_j
\end{equation*}
\item[(d)] Write $\log_m x$ for
$\log \log \cdots \log x$ with $m$ logs, where
$m$ is a nonnegative integer.  A general
transseries is obtained by substitution of some $\log_m x$
for $x$ in a log-free transseries.  Example:
\begin{equation*}
	e^{(\log\log x)^{1/2}+x} + (\log\log x)^{1/2} + x^{-2} 
\end{equation*}
A general transmonomial is obtained similarly  from
a log-free transmonomial.
\end{enumerate}
\medskip There are a few additional features in the development, as
we will see in Section~\ref{startformal}.  But
for now let's proceed to some \textit{examples.}
Computations with transseries can seem natural in many cases,
even without the technical definitions.  And---as with
generating functions---even if they do not converge.

\begin{ex}
Let us multiply $A = x-1$ times
$B = \sum_{j=0}^\infty x^{-j}$.
\begin{align*}
	(x& - 1)(1+x^{-1}+x^{-2}+x^{-3}+\dots)
	\\&= x\cdot(1+x^{-1}+x^{-2}+x^{-3}+\dots)
	- 1\cdot(1+x^{-1}+x^{-2}+x^{-3}+\dots)
	\\&= x + 1 + x^{-1} + x^{-2} + \cdots
	 -1 -x^{-1} -x^{-2} -x^{-3} -\cdots
	 \\&= x .
\end{align*}
\end{ex}

\begin{ex}
Both transseries
\begin{equation*}
	S = \sum_{j=0}^\infty j! x^{-j},\qquad
	T = \sum_{j=0}^\infty (-1)^j j! x^{-j}
\end{equation*}
are divergent.  For the product: the combinatorial identity
\begin{equation*}
	\sum_{j=0}^n (-1)^j j! (n-j)! =
	\begin{cases}
	 \displaystyle\frac{(n+1)!}{1+n/2}&\quad n \text{ even}
	 \\
	 0&\quad n \text{ odd} .
	\end{cases}
\end{equation*}
means that
\begin{equation*}
	ST = \sum_{j=0}^\infty \frac{(2j+1)!}{j+1}\,x^{-2j} .
\end{equation*}
\end{ex}

\begin{ex}
Now consider
\begin{equation*}
	U = \sum_{j=1}^\infty j e^{-jx},\qquad
	V = \sum_{k=0}^\infty x^{-k} .
\end{equation*}
When $UV$ is multiplied out, each monomial $x^{-k}e^{-jx}$
occurs only once, so our result is a transseries whose
support has order type $\omega^2$.
\begin{equation*}
	UV=\sum_{j=1}^\infty\left(\sum_{k=0}^\infty j x^{-k} e^{-jx}\right) .
\end{equation*}
(For an explanation of \Def{order type},
see \cite[p.~27]{hewitt} or \cite[p.~127]{suppes}
or even \cite{wiki}.)
\end{ex}

\begin{ex}\label{estart}
Every nonzero transseries has a multiplicative inverse.
What is the inverse of $e^x+x$?  Use the Taylor
series for $1/(1+z)$ like this:
\begin{align*}
	\big(e^x+x\big)^{-1} &= \big(e^x(1+xe^{-x})\big)^{-1}
	= e^{-x}\sum_{j=0}^\infty (-1)^j (xe^{-x})^j
	\\&=
	\sum_{j=0}^\infty (-1)^j x^j e^{-(j+1)x} .
\end{align*}
\end{ex}

\begin{ex}
The hyperbolic sine is a two-term transseries,
$\sinh x =(1/2)e^{x}-(1/2)e^{-x}$.  Let's compute
its logarithm.  Use the Taylor series for $\log(1-z)$.
\begin{equation*}
	\log(\sinh x) =
	\log\left(\frac{e^x}{2}(1-e^{-2x})\right)
	= x - \log 2 - \sum_{j=1}^\infty \frac{e^{-2jx}}{j} .
\end{equation*}
Wasn't that easy?
\end{ex}

\begin{ex}
How about the inverse of
\begin{equation*}
	T = \sum_{j=0}^\infty j! x^{-j-1}
	= x^{-1} + x^{-2} + 2x^{-3} + 6x^{-4} + 24 x^{-5} +\cdots \;?
\end{equation*}
We can compute as many terms as we want, with enough effort.
First, $T = x^{-1} (1+V)$, where
$V = x^{-1}+2x^{-2}+6x^{-3}+24x^{-4}+\cdots$ is small.
So
\begin{align*}
	T^{-1} &= (x^{-1})^{-1}(1+V)^{-1}
	= x\big(1-V+V^2-V^3+V^4-\dots\big)
	\\&= 
	x\bigg[1 - (x^{-1}+2x^{-2}+6x^{-3}+24x^{-4}+\dots)
	\\& \qquad+(x^{-1}+2x^{-2}+6x^{-3}+\dots)^2
	\\& \qquad-(x^{-1}+2x^{-2}+\dots)^3
	+(x^{-1}+\dots)^4+\dots\bigg]
	\\&= 
	x-1-x^{-1}-3x^{-2}-13x^{-3} 
	+\cdots .
\end{align*}
Searching the On-Line Encyclopedia of Integer Sequences \cite{sloane}
shows that these coefficients are sequence A003319.
\end{ex}

\begin{ex}\label{eend}
Function $x e^x$ has compositional inverse known
as the \Def{Lambert W function}.  So $W(x) e^{W(x)} = x$.
The transseries is:
\begin{align*}
W(x) =
& \log x  -\log \log x +
\frac {\log \log x }{\log x}
 +\frac{(\log  \log  x)^2}{2(\log x)^2}
 -\frac{\log \log x}{(\log x)^2}
 +\frac{(\log \log x)^3}{3(\log  x)^3}
 \\& -\frac{3( \log \log  x)^2}{2(\log  x)^3}
 +\frac{\log \log x }{(\log  x)^3}
 +\frac{(\log \log x)^4}{4(\log x)^4}
 -\frac{11(\log \log x)^3}{6(\log x)^4}
 +\cdots
\end{align*}
We will see below (Problem~\ref{lambertseries}) how to compute this.
But for now, let's see how to compute
$e^{W(x)}$.  The two terms $\log x$ and
$\log \log x$ are large, the rest is small.
If $W(x) = \log x - \log \log x + S$, then
\begin{equation*}
	e^{W(x)} = e^{\log x} e^{-\log \log x} e^S
	= \frac{x}{\log x}\left(\sum_{j=0}^\infty \frac{S^j}{j!}\right) .
\end{equation*}
Then put in $S = \log\log x/\log x + \cdots$, as many terms
as needed, to get
\begin{equation*}
	e^{W(x)} = \frac{x}{\log x} + \frac{x \log\log x}{(\log x)^2}
	+\frac{x (\log\log x)^2}{(\log x)^3}-\frac{x \log\log x}{(\log x)^3}
	+\cdots .
\end{equation*}
This is $e^{W(x)}$.
Now we can multiply this by the original $W$:
\begin{equation*}
	W(x) e^{W(x)} = x + \cdots
\end{equation*}
where the missing terms are of order higher than computed.
In fact, the claim is that all higher terms cancel.
\end{ex}

\begin{re}
By a general result of van den Dries--Macintyre--Marker
(3.12 and 6.30 in \cite{DMM}), there
exists a coherent way to associate a transseries expansion at $+\infty$
to every function $(a,+\infty) \to \R$ (where $a \in \R$) which,
like the functions considered in Examples \ref{estart} to~\ref{eend},
is \Def{definable} (in the sense of mathematical logic) from
real constants, addition, multiplication, and $\exp$.
\end{re}

\subsection*{\'Ecalle--Borel Summation}
There is a system to assign real functions to many transseries.
It is a vast generalization of the classical Borel summation
method.  Here we will consider transseries only as formal objects,
for the most part, but I could not resist including a few remarks
on summation.

The basic Borel summation works like this:  The Lapace transform
$\L$ is defined by
\begin{equation*}
	\L[F](x) = \int_0^\infty e^{-xp} F(p)\,dp ,
\end{equation*}
when it exists.  The inverse Laplace transform, or Borel
transform, will be written $\Bor$, so that
$\Bor[f] = F$ iff $\L[F] = f$.
The composition
$\L\Bor$ is an ``isomorphism'' in the sense that it
preserves ``all operations''---whatever that
means; perhaps in the wishful sense.  In fact,
in some cases even if $f$ is merely a formal series
(a divergent series), still $\L\Bor[f]$ yields
an actual function.  If so, that is the \Def{Borel sum}
of the series.

We will use variable $x$ in physical space, and variable $p$
in Borel space.  Then compute $\L[p^n] = n!x^{-n-1}$
for $n \in \N$, so
$\Bor[x^{-j}] = p^{j-1}/(j-1)!$ for integers $j\ge 1$.

\begin{ex}
Borel summation works on the series
$f=\sum_{j=0}^\infty 2^j x^{-j}$.  (Except for the first
term---no delta functions here.)
Write $f =1+g$.
First $\Bor[g] = \sum_{j=1}^\infty 2^jp^{j-1}/(j-1)! = 2e^{2p}$.
Then
\begin{equation*}
	\L\Bor[g](x) = \int_0^\infty 2 e^{-xp} e^{2p}\,dp = \frac{2}{x-2} .
\end{equation*}
Adding the $1$ back on, we conclude that the sum of the series
should be
\begin{equation*}
	1+\frac{2}{x-2} = \frac{x}{x-2} = \frac{1}{1-2/x}
\end{equation*}
as expected.

Of course the formal series $f=\sum_{j=0}^\infty 2^j x^{-j}$
satisfies $f\cdot(1-2/x) = 1$.  So if $\L\Bor$
is supposed to preserve all operations then
there is no other sum possible.
\end{ex}

\begin{ex}\label{eulerborel}
Consider Euler's series
$f = \sum_{j=0}^\infty (-1)^j j! x^{-j-1}$, a series
that diverges for all $x$.  So we want:
$\Bor[f] = \sum_{j=0}^\infty (-1)^j p^j = 1/(1+p)$.  This expression
makes sense for all $p \ge 0$, not just the ones within the
radius of convergence.  Then $\Bor[f]$ should be $1/(1+p)$.  The
the Laplace integral converges,
\begin{equation*}
	\L\Bor[f](x) = \int_0^\infty \frac{e^{-xp}\,dp}{1+p}
	= -e^x \Ei(-x) .
\end{equation*}
This is the Borel sum of the series $f$.

Similarly, consider the series
\begin{equation*}
	g = \sum_{j=0}^\infty \frac{j!}{x^{j+1}} .
\end{equation*}
In the same way, we get
\begin{equation*}
	\L\Bor[g](x) = \int_0^\infty \frac{e^{-xp}\,dp}{1-p}
\end{equation*}
where now (because of the pole at $p=1$) this
taken as a principal value integral, and we get
$e^{-x} \Ei(x)$ as the value.
\end{ex}

Borel summation is the beginning of the story.  Much more powerful
methods have been developed.  (\'Ecalle invented most of the
techniques, then others have made them rigorous and improved them.)
To a large extent
it is known that transseries that arise (from ODEs, PDEs, difference
equations, etc.) can be summed, and much more is suspected.
This summation is virtually as faithful as
convergent summation. But the subject
is beyond the scope of this paper.  In fact, it seems that
a simple exposition is not possible with our
present understanding.
For more on summation see \cite[\S3.1]{costintop},
\cite{costinglobal}, \cite{costinasymptotics}.

\section{The Formal Construction}\label{startformal}
Now we come to the technical part of the paper.  
\textit{All a mathematician needs to know
in order to apply transseries.}

To do the
types of computations we have seen, a formal construction
is desirable.  It should allow not only ``formal power series,''
but also exponentials and logarithms.  In reading this,
you can note that in fact we are not really using
high-level mathematics.  

Descriptions of the system
of transseries are found, for example, in
\cite{asch, costinasymptotics, DMM, hoeven}.
But those accounts are (to a greater or lesser extent) technical
and involve jargon of the subfield.  It is hoped that
by carefully reading this section,
a reader who is not a specialist will be
able to understand the simplicity of
the construction.  Some details
are not checked here, especially the tedious ones.

Items called \textit{Comment}, enclosed between two $\bigcirc$ signs,
are not part of the formal
construction.  They are included as illustration and
motivation.  Perhaps these commentaries cannot be completely understood
until after the formal construction has been read.

\begin{com}
Functions (or expressions) of the form $x^a e^{bx}$, where $a,b \in \R$,
are transmonomials.  (There are also many other transmonomials.
But these will be enough for most of our illustrative comments.)
We may think of the ``far larger'' relation $\fgt$ describing
relative size when $x \to +\infty$.  In particular,
$x^{a_1} e^{b_1x} \fgt x^{a_2}e^{b_2x}$ if and only if
$b_1>b_2$ or \{$b_1=b_2$ and $a_1 > a_2$\}.
\end{com}

\section*{3A\quad Multi-Indices}
\addcontentsline{toc}{section}{3A\quad Multi-Indices}%
\begin{com}
The set $\G$ of monomials is a group under multiplication.
This group (even the subgroup of monomials $x^a e^{bx}$) is not finitely
generated.  But sometimes we will want to consider
a finitely generated subgroup of $\G$.  If
$\mu_1,\cdots,\mu_n$ is a set of generators,
then the generated group is
\begin{equation*}
	\SET{\mu_1^{k_1} \mu_2^{k_2} \cdots \mu_n^{k_n} }{k_1,
	k_2, \cdots, k_n \in \Z} .
\end{equation*}
We will discuss the use of multi-indices
$\bk = (k_1,k_2,\dots,k_n)$ so that later
$\mu_1^{k_1} \mu_2^{k_2} \cdots \mu_n^{k_n}$ can be abbreviated
$\bmu^\bk$ and save much writing.

It does no harm to omit the group identity $1$ from a list
of generators; replacing some generators by their inverses,
we may assume the generators $\mu_j$ are all small:
$\mu_j \fst 1$.  (We will think of these as ``ratios'' between
one term of a series and the next.  A \Def{ratio set} is a finite
set of small monomials.)
Then the correspondence between multi-indices $\bk$
and monomials $\bmu^\bk$ reverses
the ordering.  (That is, if $\bk > \bp$, then
$\bmu^\bk \fst \bmu^\bp$.)
This means terminology that seems right on one
side may seem to be backward on the other side.
Even with conventional asymptotic series, larger terms
are written to the left, smaller terms to the right,
reversing the convention for a number line.
\end{com}

Begin with a positive integer $n$.  The set $\Z^n$ of
$n$-tuples of integers is a group under componentwise addition.
For notation---avoiding subscripts, since we want
to use subscripts for many other things---if $\bk \in \Z^n$
and $1 \le i \le n$, write $\bk[i]$ for the $i$th component
of $\bk$.  The partial order $\le$ is defined by:
$\bk \le \bp$ iff $\bk[i] \le \bp[i]$ for all $i$.  And
$\bk < \bp$ iff $\bk \le \bp$ and $\bk \ne \bp$.
Element $\0 = (0,0,\cdots,0)$ is the identity for addition.

\begin{de} 
For $\bm \in \Z^n$, define $\J_\bm = \SET{\bk \in \Z^n}{\bk \ge \bm}$.
\end{de}

\begin{com}
For example
$\J_{(-1,2)} = \SET{(k,l) \in \Z^2}{k \ge -1, l \ge 2}$.
The sets $\J_\bm$ will be used below (Definition~\ref{def.grid}) to define
``grids'' of monomials.  
If $\mu_1 = x^{-1}$ and $\mu_2 = e^{-x}$ are the ratios making
up the ratio set $\bmu$,
then
\begin{equation*}
	\SET{\bmu^\bk}{\bk \in \J_{(-1,2)}} =
	\SET{x^{-k}e^{-lx}}{k \ge -1, l \ge 2}
\end{equation*}
is the corresponding grid.
\end{com}

Write $\N = \{0,1,2,3,\cdots\}$ including $0$.  The subset
$\N^n$ of $\Z^n$ is closed under addition.
Note $\J_\bm$ is the translate of $\N^n$ by $\bm$.
That is, $\J_\bm = \SET{\bk+\bm}{\bk \in \N^n}$.
And $\N^n = \J_\0$.
Translation preserves order.

The next three propositions explain that
the set $\J_\bm$ is \Def{well-partially-ordered}
(also called \Def{Noetherian}).
These three---which are collectively known as
``Dickson's Lemma''---are the main reason why many an algorithm
in computer algebra (Gr\"obner bases) terminates.
Equivalent properties and the usefulness in power series rings
are discussed in Higman \cite{higman}
and Erd\H{o}s \& Rado (cited in \cite{higman}).

\begin{pr}\label{wellpartially}
If $\BE \subseteq \J_\bm$ and $\BE \ne \emptyset$, then there is
a minimal element:
$\bk_0 \in \BE$ and $\bk < \bk_0$ holds for no
element $\bk \in \BE$.
\end{pr}
\begin{proof}
Because translation preserves order, it suffices to
do the case of $\J_\0 = \N^n$.
First, $\SET{\bk[1]}{\bk \in \BE}$ is a nonempty subset
of $\N$, so it has a least element, say $m_1$.  Then
$\SET{\bk[2]}{\bk \in \BE, \bk[1]=m_1}$ is a nonempty subset
of $\N$, so it has a least element, say $m_2$.
Continue.  Then $\bk_0 = (m_1,\cdots, m_n)$ is minimal
in $\BE$.
\end{proof}

\begin{pr}\label{increasing}
Let $\BE \subseteq \J_\bm$ be infinite.  Then there is
a sequence $\bk_j \in \BE$, $j \in \N$, with
$\bk_0 < \bk_1 < \bk_2 < \cdots$.
\end{pr}
\begin{proof}
It is enough to do the case $\N^n$.
The proof is by induction on $n$---it is true for $n=1$.
Assume $n \ge 2$.
Define the set $\widetilde{\BE} \subseteq \Z^{n-1}$ by
\begin{equation*}
	\widetilde{\BE} = \SET{(\bk[1],\bk[2],\cdots,\bk[n-1])}{\bk \in \BE} .
\end{equation*}
\textit{Case 1.\/} $\widetilde{\BE}$ is finite.  Then for some
$\bp \in \widetilde{\BE}$, the set
\begin{equation*}
	\BE' = \SET{k \in \N}{(\bp[1],\cdots,\bp[n-1],k) \in \BE}
\end{equation*}
is infinite.
Choose an increasing sequence $k_j \in \BE'$ to get the
increasing sequence in $\BE$.

\textit{Case 2.\/} $\widetilde{\BE}$ is infinite.  By the
induction hypothesis, there
is a strictly increasing sequence $\bp_j \in \widetilde{\BE}$.
So there is a sequence $\bk_j \in \BE$ that is increasing
in every coordinate except possibly the last.  If some
last coordinate occurs infinitely often, use it to get
an increasing sequence in $E$.  If not, choose a subsequence
of these last coordinates that increases.
\end{proof}

\begin{pr}\label{magfin}
Let $\BE \subseteq \J_\bm$.  Then the set $\Min \BE$ of all minimal
elements of $\BE$ is finite.  For every $\bk \in \BE$,
there is $\bk_0 \in \Min \BE$ with $\bk_0 \le \bk$.
\end{pr}
\begin{proof}
No two minimal elements are comparable, so
$\Min \BE$ is finite by Proposition~\ref{increasing}.
If $\BE = \emptyset$, then $\Min \BE = \emptyset$ vacuously
satisfies this.  Suppose $\BE \ne \emptyset$.  Then
$\Min \BE \ne \emptyset$ satisfies the required conclusion
by Proposition~\ref{wellpartially}.
\end{proof}

\subsection*{Convergence of sets}

Write $\triangle$ for the symmetric difference operation on
sets.  We will define convergence of a sequence of sets
$\BE_j \subseteq \Z^n$ (or indeed any infinite collection $(\BE_i)_{i \in I}$
of sets).

\begin{de}
Let $I$ be an infinite index set, and for each $i \in I$, let
$\BE_i \subseteq \Z^n$ be given.  We say the family $(\BE_i)_{i \in I}$
is \Def{point-finite} iff each $\bp \in \Z^n$
belongs to $\BE_i$ for only finitely many $i$.
Let $\bm \in \Z^n$.  We write
$\BE_i \mto \emptyset$ iff $\BE_i \subseteq \J_\bm$ for all
$i$ and $(\BE_i)$ is point-finite.  We write
$\BE_i \to \emptyset$ iff there exists $\bm$ such that
$\BE_i \mto \emptyset$.  Furthermore, write
$\BE_i \mto \BE$ iff $\BE_i \subseteq \J_\bm$ for all
$i$ and $\BE_i \triangle \BE \mto \emptyset$;
and write $\BE_i \to \BE$ iff $\BE_i \mto \BE$ for some $\bm$.
\end{de}

\begin{com}
Examples in $\Z = \Z^1$.  Let
$\BE_i = \{i, i+1, i+2, \cdots \}$ for $i \in \N$.
Then the sequence $\BE_i$ is point-finite.
And $\BE_i \to \emptyset$.  But let
$\BF_i = \{-i\}$ for $i \in \N$.  Again the sequence
$\BF_i$ is point-finite, but there is no $\bm$ with
$\BF_i \subseteq \J_\bm$ for all $i$, so
$\BF_i$ does not converge in this sense.
\end{com}

This type of convergence is metrizable when restricted to any
$\J_\bm$.

\begin{no}\label{Znmetric}
For $\bk = (k_1,k_2,\cdots,k_n)$, define
$|\bk| = k_1 + k_2 + \dots + k_n$.
\end{no}

\begin{pr}
Let $\bm \in \Z^n$.  For $\BE, \BF \subseteq \J_\bm$, define
\begin{equation*}
	d(\BE,\BF) = \sum_{\bk \in \BE\triangle \BF} 2^{-|\bk|}.
\end{equation*}
Then for any sets $\BE_i \subseteq \J_\bm$, we have
$\BE_i \to \BE$ if and only if $d(\BE_i,\BE) \to 0$.
And $d$ is a metric on subsets of $\J_\bm$.
\end{pr}

\section*{3B\quad Hahn Series}
\addcontentsline{toc}{section}{3B\quad Hahn Series}%
We begin with an ordered abelian group $\MM$, called
the \Def{monomial group} (or \Def{valuation group}).
By ``ordered'' we mean totally ordered or linearly ordered.
The operation is written multiplicatively, the identity is $1$,
the order relation is $\fgt$ and read ``far larger than''.  This
is a ``strict'' order relation; that is, $\g \fgt \g$ is false.
An element $\g \in \MM$ is called \Def{large} iff $\g \fgt 1$,
and \Def{small} if $\g \fst 1$.
[We will use Fraktur letters:
lower case for monomials and upper case for sets of monomials.]

\begin{com}
The material in Subsections~\SUBSECformalseries\ and~\SUBSECgenerators\ 
will apply to any ordered abelian group $\MM$.
Later (Subsections~\SUBSECtsinf\ and~\SUBSECwithlogs) we
will construct the particular group that will
specialize this general construction into the transseries
construction.  Comments will use the group of monomials
$x^a e^{bx}$ discussed above.
\end{com}

We use the field $\R$ of real numbers for values.
Write $\R^\MM$ for the set of functions $T \takes \MM \to \R$.
For $T \in \R^\MM$ and $\g \in \MM$, we will use square brackets
$T[\g]$ for the value of $T$ at $\g$, because later we will want to
use round brackets $T(x)$ in another more common sense.

\begin{de}
The \Def{support} of a function $T \in \R^\MM$
is
\begin{equation*}
	\supp T = \SET{\g \in \MM}{T[\g] \ne 0}.
\end{equation*}
Let $\AA \subseteq \MM$.  We say $T$ is \Def{supported by} $\AA$
if $\supp T \subseteq \AA$.
\end{de}

\begin{no}
In fact, $T$ will usually be written as a \Def{formal combination of
group elements}.  That is:
\begin{equation*}
	T = \sum_{\g \in \AA} a_\g \g,\qquad a_\g \in \R
\end{equation*}
will be used for the function $T$ with $T[\g] = a_\g$ for
$\g \in \AA$ and $T[\g] = 0$ otherwise.
The set $\AA$ might or might not be the actual support of $T$.
Accordingly, such $T$ may be called a \Def{Hahn series}
or \Def{generalized power series}.
\end{no}

\begin{de}
If $c \in \R$, then $c\,1 \in \R^\MM$ is called a \Def{constant}
and identified with $c$.  (That is, $T[1] = c$ and
$T[\g] = 0$ for all $\g \ne 1$.)
If $\m \in \MM$, then $1\,\m \in \R^\MM$ is called a \Def{monomial}
and identified with $\m$.  (That is, $T[\m] = 1$ and $T[\g]=0$
for all $\g \ne \m$.)
\end{de}

In all cases of interest to us, the support will be
\Def{well ordered} (according to the converse of $\fgt$).  That
is, for all $\AA \subseteq \supp(T)$, if $\AA \ne \emptyset$,
it has a maximum: $\m \in \AA$ such that for all $\g \in \AA$,
if $\g \ne \m$, then $\m \fgt \g$.

\begin{pr}
Let $\AA \subseteq \MM$ be well ordered for the converse of \,$\fgt$.
Every infinite subset in $\AA$ contains an infinite strictly
decreasing sequence $\g_1 \fgt \g_2 \fgt \cdots$.  There is no
infinite strictly increasing sequence in $\AA$.
\end{pr}

\begin{de}\label{magdef}
Let $T \ne 0$ be
\begin{equation*}
	T = \sum_{\g \in \AA} a_\g \g,\qquad a_\g \in \R ,
\end{equation*}
with $\m \in \AA$,
$\m \fgt \g$ for all other $\g \in \AA$, and $a_{\m} \ne 0$.
Then the \Def{magnitude} of $T$ is $\mag T = \m$,
the \Def{leading coefficient} of $T$ is $a_{\m}$, and
the \Def{dominance} of $T$ is $\dom T = a_{\m}\m$.
We say $T$ is \Def{positive} if $a_{\m} > 0$ and write $T > 0$.
We say $T$ is \Def{negative} if $a_{\m} < 0$ and write $T < 0$.
We say $T$ is \Def{small} if $\g \fst 1$ for all $\g \in \supp T$
(equivalently: $\mag T \fst 1$  or $T = 0$).
We say $T$ is \Def{large} if $\mag T \fgt 1$.
We say $T$ is \Def{purely large} if $\g \fgt 1$ for all $\g \in \supp T$.
(Because of the standard empty-set conventions:
$0$, although not large, is purely large.)
\end{de}

\begin{re}
Alternate terminology \cite{hoeven,schmeling}:
magnitude = leading monomial;
dominance = leading term;
large = infinite;
small = infinitesimal.
\end{re}

\begin{com}
Let $A = -3e^x + 4 x^2$.  Then $\mag A = e^x$,
$\dom A = -3e^x$, $A$ is negative, $A$ is large,
$A$ is purely large.
\end{com}

\begin{de}
\Def{Addition} is defined by components:  $(A+B)[\g] = A[\g] + B[\g]$.
\Def{Constant multiples} $c A$ are also defined by components.
\end{de}

\begin{re}
The union of two well ordered sets is well ordered.  So
if $A, B$ each have well ordered support, so does $A+B$.
\end{re}

\begin{no}
We say $S > T$ if $S-T > 0$.  For nonzero $S$ and $T$
we say $S \fgt T$ (read $S$ is \Def{far larger than} $T$) iff
$\mag S \fgt \mag T$;
we say $S \cto T$ (read $S$ is \Def{comparable to} $T$
or $S$ has the \Def{same magnitude} as $T$)
iff $\mag S = \mag T$;
and we say $S \ato T$ (read $S$ is \Def{asymptotic to} $T$)
iff $\dom S = \dom T$.  Write $S \fgte T$ iff
$S \fgt T$ or $S \cto T$.
\end{no}

\begin{comnoc}
Examples:
\begin{align*}
	-3e^x + 4 x^2 &< x^9 ,\\
	-3e^x + 4 x^2 &\fgt x^9 ,\\
	-3e^x + 4 x^2 &\cto 7e^x+x^9 ,\\
	-3e^x + 4 x^2 &\ato -3e^x+x^9 .
	\qquad\qquad\qquad\qquad\bigcirc
\end{align*}
\end{comnoc}

\subsection*{The Two Canonical Decompositions}

\begin{pr}[Canonical Additive Decomposition]\label{C_add}
Every $A$ may be written uniquely in the form
$A = L + c + S$, where $L$ is purely large,
$c$ is a constant, and $S$ is small.
\end{pr}

\begin{re}
Terminology: $L$ is the \Def{purely large part},
$c$ is the \Def{constant term}, and $S$ is the
\Def{small part} of $A$.
\end{re}

\begin{de}
\Def{Multiplication} is defined by convolution (as suggested by the
formal sum notation).
\begin{align*}
	& \sum_{\g\in\MM} a_\g \g \cdot \sum_{\g\in\MM} b_\g \g =
	\sum_{\g\in\MM} \left(\sum_{\m \n=\g} a_{\m} b_{\n}\right)\;\g ,
	\\ & \text{or }\qquad
	(AB)[\g] = \sum_{\m\n = \g} A[\m] B[\n] .
\end{align*}
\end{de}

Products are defined at least for $A, B$ with well ordered support.

\begin{pr}\label{Ewellproduct}
If $\AA_1, \AA_2 \subseteq \MM$ are well ordered sets
{\rm (}for the converse of $\fgt${\rm )}, then
$\AA = \SET{\g_1 \g_2}{\g_1 \in \AA_1, \g_2 \in \AA_2}$
is also well ordered.  For every $\g \in \AA$, the set
\begin{equation*}
	\SET{(\g_1,\g_2)}{\g_1 \in \AA_1, \g_2 \in \AA_2, \g_1 \g_2 = \g}
\end{equation*}
is finite.
\end{pr}
\begin{proof}
Let $\BB \subseteq \AA$ be nonempty.  Assume $\BB$
has no greatest element.  Then there exist sequences $\m_j \in \AA_1$
and $\n_j \in \AA_2$ with $\m_j \n_j \in \BB$ and
$\m_1\n_1 \fst \m_2 \n_2 \fst \cdots$.
Because $\AA_1$ is well ordered, taking a subsequence we
may assume $\m_1 \fgte \m_2 \fgte \cdots$.  But then
$\n_1 \fst \n_2 \fst \cdots$, so $\AA_2$ is not well ordered.

Suppose $(\g_1,\g_2),(\m_1,\m_2) \in \AA_1 \times \AA_2$ with
$\g_1 \g_2 = \g = \m_1 \m_2$.  If $\g_1 \ne \m_1$, then $\g_2 \ne \m_2$.
If $\g_1 \fgt \m_1$, then $\g_2 \fst \m_2$.  Any infinite subset of
a well ordered set contains an infinite strictly decreasing sequence,
but the other well ordered set contains no infinite strictly
increasing sequence.
\end{proof}

\begin{pr}
The set of all $T \in \R^\MM$ with well ordered support is an
(associative, commutative) algebra over $\R$ with the operations defined above.
\end{pr}

There are a lot of details to check.   In fact this
is a field \cite[p.~276]{cohn}, but we won't need
that result.  This goes back to H.~Hahn, 1907 \cite{hahn}.

\begin{pr}[Canonical Multiplicative Decomposition]\label{C_mult}
Every nonzero $T \in \R^\MM$ with well ordered support may be
written uniquely in the form
$T = a \cdot\g \cdot(1+S)$ where $a$ is nonzero real,
$\g \in \MM$, and $S$ is small.
\end{pr}

\begin{com}
$-3e^x + 4 x^2 = -3\cdot e^x\cdot \big(1-(4/3)x^2e^{-x}\big)$.
\end{com}

\begin{pr}
The set of all purely large $T$
{\rm (}including $0${\rm )} is a group under
addition.  The set of all small $T$ is a group under addition.
The set of all purely large $T$
{\rm (}with well ordered support{\rm )}
is closed under multiplication.  The set of all small $T$ {\rm (}with
well ordered support{\rm )} is closed under multiplication.
\end{pr}

\section*{3C\quad Grids}
\addcontentsline{toc}{section}{3C\quad Grids}%
Some definitions will depend on a finite set of \Def{ratios}
(or \Def{generators}).
We will keep track of the set of ratios more than is customary.
But it is useful for the proofs, and especially for the
Fixed-Point Theorem~\ref{costinfixed}.

Write $\Msmall = \SET{\g \in \MM}{\g \fst 1}$.
A \Def{ratio set} (or \Def{generating set})
is a finite subset $\bmu \subset \Msmall$.
We will use bold Greek for ratio sets.
If convenient, we may number the elements of $\bmu$ in order,
$\mu_1 \fgt \mu_2 \fgt \cdots \fgt \mu_n$ and then consider
$\bmu$ an ordered $n$-tuple.

\begin{no}
Let $\bmu = \{\mu_1,\cdots,\mu_n\}$.  For any mul\-ti-in\-dex
$\bk = (k_1,\cdots,k_n)$ $\in \Z^n$, define $\bmu^\bk = \mu_1^{k_1}\cdots\mu_n^{k_n}$.
\end{no}

If $\bk > \bp$, then $\bmu^\bk \fst \bmu^\bp$.
Also $\bmu^\0 = 1$.  If $\bk > \0$ then $\bmu^\bk \fst 1$
(but not in general conversely).

\begin{com}
Let $\bmu = \{x^{-1},e^{-x}\}$.  Then
$1 \fgt \mu_1^{-1} \mu_2 = x e^{-x}$, even though $(-1,1) \not> (0,0)$.
\end{com}

\begin{de}\label{def.grid}
For ratio set $\bmu = \{\mu_1,\cdots,\mu_n\}$,
let $\GRID^{\ebmu} = \SET{\bmu^\bk}{\bk \in \Z^n}$, the
subgroup generated by $\bmu$.  If
$\bm \in \Z^n$, then define a subset of $\GRID^{\ebmu}$ by
\begin{equation*}
	\GRID^{\ebmu,\bm} = \SET{\bmu^\bk}{\bk \in \J_\bm}
	= \SET{\bmu^\bk}{\bk \ge \bm} .
\end{equation*}
The sets $\GRID^{\ebmu,\bm}$ are called \Def{grids}.
A Hahn series $T \in \R^\MM$ supported by some grid is
said to be \Def{grid-based}.  A set $\AA \subseteq \MM$
is called a \Def{subgrid} if $\AA \subseteq \GRID^{\ebmu,\bm}$
for some $\bmu, \bm$.
\end{de}

\begin{pr}
Let $\W$ be the set of all subgrids.
\begin{enumerate}
\item[{\rm (a)}] $\GRID^{\ebmu,\bm}$ is well ordered {\rm (}by
the converse of $\fgt${\rm )}.
\item[{\rm (b})] If $\bmu \subseteq \tbmu$,
then $\GRID^{\ebmu,\bm} \subseteq
\GRID^{\tebmu,\tilde{\bm}}$ for some $\widetilde{\bm}$.
\item[{\rm (c)}] If $\AA, \BB \in \W$, then $\AA \cup \BB \in \W$.
\item[{\rm (d)}] If $\AA, \BB \in \W$, then $\AA \cdot \BB \in \W$,
where $\AA \cdot \BB := \SET{\fa\fb}{\fa \in \AA, \fb \in \BB}$.
\end{enumerate}
\end{pr}
\begin{proof}
(a) Let $\BB \subseteq \GRID^{\ebmu,\bm}$ be nonempty.
Define $\BE = \SET{\bk \in \J_\bm}{\bmu^\bk \in \BB}$.
Then the set $\Min \BE$ of minimal elements of $\BE$ is finite.  So
the greatest element of $\BB$ is
$\max\SET{\bmu^\bk}{\bk \in \Min \BE}$.

(b) Insert $0$s for the extra entries of $\widetilde{\bm}$.

(c) Use the union of the two $\bmu$s and the minimum
of the two $\bm$s.

(d) Use the union of the two $\bmu$s and the sum
of the two $\bm$s.
\end{proof}

\begin{re}
By (c) and (d), if $S, T \in \R^\MM$ each have support in
$\W$, then $S+T$ and $S T$ also have support in $\W$.
\end{re}

\begin{re}
Write $\bmu = \{\mu_1,\mu_2,\cdots,\mu_n\}$ and
$\bm = (m_1,m_2,\cdots,m_n)$.
Saying $T$ is supported by the grid $\GRID^{\ebmu,\bm}$ means
that $T = \sum c_\bk \bmu^\bk$ is a
one-sided multiple \Def{Laurent series} in the symbols $\bmu$:
\begin{equation*}
	\sum_{k_1=m_1}^\infty \sum_{k_2=m_2}^\infty ... \sum_{k_n=m_n}^\infty 
	c_{k_1k_2\dots k_n} \mu_1^{k_1} \mu_2^{k_2}\cdots \mu_n^{k_n} .
\end{equation*}
\end{re}

\begin{com}
This is one advantage of the grid-based approach.  We consider
series only of this ``multiple Laurent series'' type.
We do not have to contemplate series supported by abstract ordinals,
something that may be considered ``esoteric''---at least
by beginners.
\end{com}

\begin{de}\label{def:gridseries}
Let $\bmu = \{\mu_1,\cdots,\mu_n\} \subseteq \Msmall$
be a ratio set
and $\bm \in \Z^n$.
The set of
series supported by the grid $\GRID^{\ebmu,\bm}$ is
\begin{equation*}
	\T^{\ebmu,\bm} = \SET{T \in \R^\MM}{\supp T \subseteq \GRID^{\ebmu,\bm}}.
\end{equation*}
The set of $\bmu$-based series is
\begin{equation*}
	\T^{\ebmu} = \bigcup_{\bm \in \Z^n} \T^{\ebmu,\bm} .
\end{equation*}
The set of grid-based series is
\begin{equation*}
	\R\lbb\MM\rbb = \bigcup_{\ebmu} \T^{\ebmu} .
\end{equation*}
In this union, all finite sets $\bmu \subset \Msmall$ are allowed, and all
values of $n$ are allowed.  But each individual series is
supported by a grid $\GRID^{\ebmu,\bm}$
generated by one finite set $\bmu$.
\end{de}

If $\bmu \subseteq \tbmu$, then $\T^{\ebmu} \subseteq \T^{\tebmu}$
in a natural way.  If $\MM$ is a subgroup of $\widetilde{\MM}$ and
inherits the order, then $\R\lbb\MM\rbb \subseteq \R\lbb\,\widetilde{\MM}\,\rbb$
in a natural way.

\begin{com}
The series
\begin{equation*}
	\sum_{j=1}^\infty x^{1/j} = x^1 + x^{1/2} + x^{1/3} + x^{1/4} +\cdots ,
\end{equation*}
despite having well ordered support, does not belong to $\R\lbb\MM\rbb$.
It is not grid-based.
\end{com}

\begin{com}
The correspondence $\bk \mapsto \bmu^\bk$ may fail to be injective.
Let $\bmu = \{x^{-1/3},x^{-1/2}\}$.  Then $\mu_1^3 = \mu_2^2$.
\end{com}

\begin{pr}\label{finitetoone}
Given $\bmu, \bm, \g$, there are only finitely many
$\bk \in \J_\bm$ with $\bmu^\bk = \g$.
\end{pr}
\begin{proof}
Suppose there are infinitely many $\bk \in \J_\bm$ with
$\bmu^\bk = \g$.  By Proposition~\ref{increasing}, this includes
$\bk_1 < \bk_2$.  But then $\bmu^{\bk_1} \fgt \bmu^{\bk_2}$,
so they are not both equal to $\g$.
\end{proof}

The map $\bk \mapsto \bmu^\bk$ might not be one-to-one,
but it is finite-to-one.  So:  if
$T_i \in \T^{\ebmu,\bm}$ for all $i \in I$,
and $\BE_i = \SET{\bk \in \J_\bm}{\bmu^\bk \in \supp T_i}$,
then $(\supp T_i)$ is point-finite if and only if
$(\BE_i)$ is point-finite.  We may sometimes say
a family $(T_i)$ is \Def{point-finite} when the family
$(\supp T_i)$ of supports is point-finite.

\subsection*{Manifestly Small}

\begin{de} If $\g$ may be
written in the form $\bmu^\bk$ with $\bk > \0$,
then $\g$ is \Def{$\bmu$-small}, written $\g \fst^{\ebmu} 1$.
[For emphasis, \Def{manifestly $\bmu$-small}.]
If every $\g \in \supp T$ is $\bmu$-small, then we say $T$
is \Def{$\bmu$-small}, written $T \fst^{\ebmu} 1$.
\end{de}

\begin{com}\label{ex:musmall}
Let $\bmu = \{x^{-1},e^{-x}\}$.  Then $\g = x e^{-x}$ is
small, but not $\bmu$-small.  For
$T = x^{-1} + xe^{-x}$ we have $\mag T \fst^{\ebmu} 1$ but not
$T \fst^{\ebmu} 1$.
\end{com}

\subsection*{The Asymptotic Topology}

\begin{de}\label{Easympt}
\Def{Limits} of grid-based
series.  Let $I$ be an infinite index set (such
as $\N$) and let $T_i, T \in \R\lbb\MM\rbb$ for $i \in I$.  Then:
(a)~$T_i \overset{\ebmu,\bm}{\longrightarrow} T$ means:
$\supp T_i \subseteq \GRID^{\ebmu,\bm}$
for all $i$, and the family
$\supp(T_i-T)$ is point-finite.
(b)~$T_i \muto T$ means:
there exists $\bm$ such that $T_i \overset{\ebmu,\bm}{\longrightarrow} T$.
(c)~$T_i \to T$ means: there exists $\bmu$ such that
$T_i \muto T$.  (See the ``asymptotic topology''
in \cite[\S1.2]{costintop}.)
\end{de}

\begin{com}
The sequence $(x^j)_{j \in \N}$ is point-finite,
but it does not converge to $0$
because the supports are not contained
in any fixed grid $\GRID^{\ebmu,\bm}$.
\end{com}

\begin{com}
This type of convergence is not the convergence associated
with the order.  For example, $(x^{-j})_{j \in \N} \to 0$ even
though $x^{-j} \fgt e^{-x}$ for all $j$.  Another example:
The grid-based series
$\sum_{j=0}^\infty {x^{-j}}$ is $A = (1-x^{-1})^{-1}$, even though
there are many grid-based series (for example,
$A - e^{-x}$)
strictly smaller than $A$ but strictly larger
than all partial sums $\sum_{j=0}^N {x^{-j}}$.

In fact, the order topology would have poor algebraic properties
for sequences:
For example
\begin{equation*}
	x^{-1} > x^{-2}+e^{-x} > x^{-3} > x^{-4}+e^{-x} > \cdots
\end{equation*}
(in both orderings $>$ and $\fgt$).
So in the order topology the sequences $x^{-j}$ and
$x^{-j}+e^{-x}$ should have the same limit, but their difference
does not converge to zero.
\end{com}

\begin{pr}[Continuity]\label{Emultcontin}
Let $I$ be an infinite index set, and let
$A_i, B_i \in \R\lbb\MM\rbb$ for $i \in I$.
If $A_i \to A$ and $B_i \to B$, then
$A_i + B_i \to A+B$ and $A_i B_i \to AB$.
\end{pr}
\begin{proof}
We may increase $\bmu$ and decrease $\bm$ to arrange
$A_i \overset{\ebmu,\bm}{\longrightarrow} A$ and
$B_i \overset{\ebmu,\bm}{\longrightarrow} B$ for the same
$\bmu, \bm$.  Then
$A_i + B_i \overset{\ebmu,\bm}{\longrightarrow} A+B$
and $A_i B_i \overset{\ebmu,\bp}{\longrightarrow} AB$
for $\bp = 2\bm$. To see this: let $\g \in \GRID^{\ebmu,\bp}$.
There are finitely many pairs $\m, \n \in \GRID^{\ebmu,\bk}$
such that $\m\n = \g$ (Proposition~\ref{Ewellproduct}).
So there is a single finite $I_0 \subseteq I$ outside
of which $A_i[\m] = A[\m]$ and $B_i[\n] = B[\n]$
for all such $\m,\n$.
For such $i$, we also have $(A_i B_i)[\g] = (A B)[\g]$.
\end{proof}

\begin{de}
\Def{Series} of grid-based
series.  Let $A_i, S \in \R\lbb\MM\rbb$ for $i$
in some index set $I$.  Then
\begin{equation*}
	S = \sum_{i \in I} A_i
\end{equation*}
means: there exist $\bmu$ and $\bm$ such that
$\supp A_i \subseteq \GRID^{\ebmu,\bm}$ for all $i$;
for all $\g$, the set $I_\g = \SET{i \in I}{A_i[\g] \ne 0}$ is finite;
and $S[\g] = \sum_{i \in I_\g} A_i[\g]$.
\end{de}

\begin{pr}
If $S \in \R\lbb\MM\rbb$, then the ``formal combination of group elements''
that specifies $S$ in fact converges to $S$ in this sense as well.
\end{pr}

Note we have the ``nonarchimedean'' (or ``ultrametric'') Cauchy criterion:
In the asymptotic topology, a
series $\sum A_i$ converges if and only if $A_i \to 0$.

\begin{pr}\label{limzero}
Let $S \in \T^{\ebmu}$ be $\bmu$-small.  Then $(S^j)_{j \in \N} \muto 0$.
\end{pr}
\begin{proof}
Every monomial in $\supp S$ can be written in the form
$\bmu^\bk$ with $\bk > \0$.  The product of two of these
is again one of these.
Let $\g_0 \in \MM$.  If $\g_0$ is not $\bmu$-small, then $\g_0 \in \supp(S^j)$
for no $j$.  So assume $\g_0$ is $\bmu$-small.  Then
there are just finitely many $\bp > \0$ such that
$\g_0 = \bmu^\bp$.  Let
\begin{equation*}
	N = \max\SET{|\bp|}{\bp > \0, \bmu^\bp = \g_0} .
\end{equation*}
Now let $j > N$.  Since every $\g \in \supp S$
is $\bmu^\bk$ with $|\bk| \ge 1$, we see that
every element of $\supp(S^j)$ is $\bmu^\bk$
with $|\bk| \ge j > N$.  So $\g_0 \not\in \supp(S^j)$.
This shows the family $(\supp(S^j))$ is point-finite.
\end{proof}

\begin{pr}\label{addendumsmall}
Let $\bmu \subseteq \Msmall$ have $n$ elements.
{\rm (a)}~Let $T \in \T^{\ebmu}$ be small.
Then there is a {\rm (}possibly larger{\rm )}
finite set $\tbmu \subseteq \Msmall$ such that
$T$ is manifestly $\tbmu$-small.
{\rm (b)}~Let $\bm \in \Z^n$.
There is a finite set $\tbmu \subseteq \Msmall$ such that
$\GRID^{\ebmu,\bm} \cap \Msmall \subseteq
\GRID^{\tebmu,\0}\setminus\{1\}$.
\hbox{{\rm (See 
\cite[Proposition~2.1]{hoeven},
\cite[{4.168}]{costinasymptotics}.)}}
\end{pr}
\begin{proof}
(a) follows from (b).
Let $\BE = \SET{\bk\in \J_\bm}{\bmu^\bk \fst 1}$, so that
$\GRID^{\ebmu,\bm} \cap \Msmall = \SET{\bmu^\bk}{\bk \in \BE}$.  
By Proposition~\ref{magfin}, $\Min \BE$ is finite.
Let $\tbmu = \bmu \cup \SET{\bmu^\bk}{\bk \in \Min \BE}$.
Note $\tbmu \subset \Msmall$.  It is the original
set $\bmu$ together with finitely many additional elements.
Now for any $\g \in \GRID^{\ebmu,\bm} \cap \Msmall$,
there is $\bp \in \BE$ with
$\bmu^\bp = \g$, and then there is $\bk \in \Min \BE$ with
$\bp \ge \bk$, so that $\m = \bmu^\bk \in \tbmu$
and $\g = \m \bmu^{\bp-\bk}$.
But $\bmu^{\bp-\bk}$ is
$\bmu$-small and $\m \in \tbmu$,
so $\g$ is is manifestly $\tbmu$-small.
\end{proof}

Call the set $\tbmu \setminus \bmu$ in (a) the
\Def{smallness addendum} for $T$.

\begin{com}\label{ex:musmall2}
Continue Comment~\ref{ex:musmall}: If $\bmu = \{x^{-1},e^{-x}\}$
then $xe^{-x}$ is small but not $\bmu$-small.
But if we change to $\tbmu = \{x^{-1},xe^{-x},e^{-x}\}$,
then $xe^{-x}$ is $\tbmu$-small.
\end{com}

\begin{com}
The statement like Proposition~\ref{addendumsmall}
for purely large $T$ is false.
The grid-based series
\begin{equation*}
	T = \sum_{j=0}^\infty  x^{-j} e^x
\end{equation*}
is purely large, but there is no finite set $\bmu \subseteq \Msmall$
and multi-index $\bm$ such that all $x^{-j}e^x$ have the
form $\bmu^\bk$ with $\bm \le \bk < \0$.  This is because
the set $\SET{\bk}{\bm \le \bk < \0}$ is finite.
\end{com}

\begin{pr}
Let $S \in \R\lbb\MM\rbb$ be small.  Then $(S^j)_{j \in \N} \to 0$.
\end{pr}
\begin{proof}
First, $S \in \T^{\ebmu}$ for some $\bmu$.  Then
$S$ is manifestly
$\tbmu$-small for some $\tbmu \supseteq \bmu$.
Therefore $S^j \overset{\tebmu}{\longrightarrow} 0$
by Proposition~\ref{limzero},
so $S^j \to 0$.
\end{proof}

\begin{pr}\label{series}
Let $\sum_{j=0}^\infty c_j z^j$ be a power series. {\rm (}Not
assumed to have positive
radius of convergence zero.{\rm )}  If $S$ is a small
grid-based series,
then $\sum_{j=0}^\infty c_j S^j$ converges in the asymptotic topology.
\end{pr}
\begin{proof}
Use Proposition~\ref{addendumsmall}.  We need to add the smallness
addendum of $S$ to $\bmu$ to get a set $\tbmu$
such that $\sum_{j=0}^\infty c_j S^j$ is
$\tbmu$-convergent.
\end{proof}

\begin{com}\label{ex:musmall3}
Continue Comment \ref{ex:musmall2}: If $\bmu = \{x^{-1},e^{-x}\}$
then $S = xe^{-x}$ belongs to $\T^{\ebmu}$ and is small but
the series $\sum_{j=0}^\infty S^j$
is not $\bmu$-convergent.
Increase to $\tbmu = \{x^{-1},xe^{-x},e^{-x}\}$
and then $\sum_{j=0}^\infty S^j$ is $\tbmu$-convergent.
\end{com}

\begin{pr}\label{finiteseries}
Let $S_1,\cdots,S_m$ be $\bmu$-small grid-based series and let
$p_1,\cdots$, $p_m \in \Z$.  Then the family
\begin{equation*}
	\SET{\supp \big(S_1^{j_1} S_2^{j_2} \cdots S_m^{j_m}\big)}{j_1 \ge p_1,
	\dots, j_m \ge p_m}
\end{equation*}
is point-finite.  That is, all multiple Laurent series
of the form
\begin{equation*}
	\sum_{j_1=p_1}^\infty \sum_{j_2=p_2}^\infty ... \sum_{j_m=p_m}^\infty 
	c_{j_1j_2\dots j_m} S_1^{j_1} \cdots S_m^{j_m}
\end{equation*}
are $\bmu$-convergent.
\end{pr}
\begin{proof}
An induction on $m$ shows that we may assume
$p_1 = \dots = p_m = 1$, since the series with general $p_i$
and the series with all $p_i=1$, differ from each other
by a finite number of series
with fewer summations.  So assume $p_1 = \dots = p_m = 1$.

Let $\g_0 \in \MM$.  If $\g_0$ is not $\bmu$-small,
then $\g_0 \in \supp \big(S_1^{j_1} \cdots S_m^{j_m}\big)$
for no $j_1,\cdots,j_m$.  So assume $\g_0$ is $\bmu$-small.
There are finitely many $\bk > \0$ so that
$\bmu^\bk = \g_0$.  Let
\begin{equation*}
	N = \max\SET{|\bk|}{\bk > \0, \bmu^\bk = \g_0} .
\end{equation*}
Each monomial in each $\supp S_i$ has the
form $\bmu^\bk$ with $|\bk| \ge 1$.
So if $j_1+\dots+j_m > N$, we have
$\g_0 \not\in \supp\big(S_1^{j_1} \cdots S_m^{j_m}\big)$.
\end{proof}

\begin{pr}\label{addenduminverse}
Let $A \in \T^{\ebmu}$ be nonzero.
Then there is a {\rm (}possibly larger{\rm )}
finite set
$\tbmu \subseteq \Msmall$ and $B \in \T^{\tebmu}$ such that
$B A = 1$.  The set $\R\lbb\MM\rbb$ of all grid-based
series supported by a group $\MM$ is a field.
\end{pr}
\begin{proof}
Write $A = c \bmu^\bk\,(1+S)$, as in \ref{C_mult},
$\bk \in \Z^n$.  Then the inverse $B$ is:
\begin{equation*}
	B = c^{-1} \bmu^{-\bk}\,\sum_{j=0}^\infty (-1)^j S^j .
\end{equation*}
Now $c^{-1}$ is computed in $\R$.  For the series,
use Proposition~\ref{series}.  Let $\tbmu$ be $\bmu$ plus
the smallness addendum for $S$.
\end{proof}

We will call $\tbmu \setminus \bmu$
the \Def{inversion addendum} for $A$.

\begin{com}\label{ex:musmall4}
Continue Comment~\ref{ex:musmall3}: If $\bmu = \{x^{-1},e^{-x}\}$
and $A = 1 + xe^{-x}$, then $A \in \T^{\ebmu}$.
But $A$ has no inverse in $\T^{\ebmu}$.
Increase to $\tbmu = \{x^{-1},xe^{-x},e^{-x}\}$
and then $A^{-1} \in \T^{\tebmu}$.
\end{com}

The algebra $\R\lbb\MM\rbb$ is an \Def{ordered field}:
If $S, T > 0$, then $S+T > 0$ and $ST>0$.  Also: if
$T_i > 0$ and $\sum T_i$ exists, then $\sum T_i > 0$.

\begin{com}
But: if $T_i \ge 0$, $T_i \to T$, then $T \ge 0$ need not follow.
Take $T_i = x^{-i}e^x - x$ and $T = -x$.
Also:
$S,T \fgt 1$ need not imply $S+T \fgt 1$.  For
example, $S = x$, $T = -x+e^{-x}$.
\end{com}

\section*{3D\quad Transseries for $x \to +\infty$}
\addcontentsline{toc}{section}{3D\quad Transseries for $x \to +\infty$}%
For (real, grid-based) \Def{transseries}, we define a specific ordered
group $\G$ of \Def{transmonomials} to use for $\MM$.  This is done in stages.

\begin{com}
A symbol ``$x$'' appears in the notation.  When we think
of a transseries as describing behavior as $x \to +\infty$,
then $x$ is supposed to be a large parameter.  When we write
``compositions'' involving transseries, $x$ represents the
identity function.  But usually it is just a convenient symbol.
\end{com}

\begin{no}
Group $\G_0$ is isomorphic to $\R$ with addition and the usual
ordering.  To fit our applications, we write $x^b$ for the group
element corresponding to $b \in \R$.  Then
$x^a x^b = x^{a+b}$; $x^0 = 1$; $x^{-b}$ is the inverse of $x^b$;
$x^a \fst x^b$ iff $a < b$.

Log-free transseries of \Def{height zero} are those obtained from
this group as in Definition~\ref{def:gridseries}.  Write
$\T_0 = \R\lbb\G_0\rbb$.  Then the set of purely large transseries in
$\T_0$ (including $0$) is a group under addition.
\end{no}

\begin{com}
Transseries of height zero:
\begin{equation*}
	-x^3+2x^2-x, \qquad
	\sum_{j=1}^\infty \sum_{k=1}^\infty x^{-j-k\sqrt{2}} .
\end{equation*}
The first is purely large, the second is small.
\end{com}

\begin{no}
Group $\G_1$ consists of ordered pairs $(b,L)$---but written
$x^b e^L$---where $b \in \R$ and $L \in \T_0$ is purely large.  Define
the group operation:
$(x^{b_1} e^{L_1})\,(x^{b_2} e^{L_2})
= x^{b_1+b_2}\,e^{L_1+L_2}$.  Define order lexicographically:
$(x^{b_1} e^{L_1}) \fgt (x^{b_2} e^{L_2})$ iff either
$L_1 > L_2$ or \{$L_1 = L_2$ and $b_1 > b_2$\}.
Identify $\G_0$ as a subgroup of $\G_1$, where
$x^b$ is identified with $x^b e^0$.

Log-free transseries of \Def{height} $1$ are those obtained from
this group as in Definition~\ref{def:gridseries}.  Write
$\T_1 = \R\lbb\G_1\rbb$.  We may identify $\T_0$ as a subset
of $\T_1$.  Then the set of purely large transseries in
$\T_1$ (including $0$) is a group under addition.
\end{no}

\begin{com}
Transseries of height $1$:
\begin{equation*}
	e^{-x^3+2x^2-x}, \qquad
	\sum_{j=1}^\infty x^{-j} e^x , \qquad
	x^3+e^{-x^{3/4}} .
\end{equation*}
The first is small, the second is purely large, the
last is large but not purely large.
\end{com}

\begin{no}
Suppose log-free transmonomials $\G_N$ and log-free transseries
$\T_N$ of height $N$ have been defined.
Group $\G_{N+1}$ consists of ordered pairs $(b,L)$
but written $x^b e^L$,
where $b \in \R$ and $L \in \T_N$ is purely large.  Define
the group operation:
$(x^{b_1} e^{L_1})\,(x^{b_2} e^{L_2})
= x^{b_1+b_2}\,e^{L_1+L_2}$.  Define order:
$(x^{b_1} e^{L_1}) \fgt (x^{b_2} e^{L_2})$ iff either
$L_1 > L_2$ or \{$L_1 = L_2$ and $b_1 > b_2$\}.

Identify $\G_N$ as a subgroup of $\G_{N+1}$ recursively.

Log-free transseries of height $N+1$ are those obtained from
this group as in Definition~\ref{def:gridseries}.  Write
$\T_{N+1} = \R\lbb\G_{N+1}\rbb$.  We may identify $\T_N$ as a subset
of $\T_{N+1}$.
\end{no}

\begin{com}
Height 2:
$e^{-e^x}, e^{\sum_{j=1}^\infty x^{-j} e^x}$.
\end{com}

\begin{no}
The group of log-free \Def{transmonomials} is
\begin{equation*}
	\G_{\bullet} = \bigcup_{N \in \N} \G_N .
\end{equation*}
The field of log-free \Def{transseries} is
\begin{equation*}
	\T_{\bullet} = \bigcup_{N \in \N} \T_N .
\end{equation*}
In fact, $\T_{\bullet} = \R\lbb\G_{\bullet}\rbb$ because each individual transseries is grid-based.  Any grid in $\G_\bullet$ is contained
in $\G_N$ for some $N$.
\end{no}

A ratio set $\bmu$ is \Def{hereditary} if for every
transmonomial $x^be^L$ in $\bmu$, we also have $L \in \T^{\ebmu}$.
Of course, given any ratio set $\bmu \subseteq \Gsmall$,
there is a hereditary ratio set
$\tbmu \supseteq \bmu$.  (When we add a set of generating ratios
for the exponents $L$, then sets of generating ratios for their exponents,
and so on, the process ends in finitely many steps, by induction
on the heights.  We do need to know: If $T \in \T_\bullet$,
then the union
\begin{equation*}
	\bigcup_{x^be^L \in \supp T} \supp L
\end{equation*}
is a subgrid.)
Call $\tbmu \setminus \bmu$ the \Def{heredity addendum}
of $\bmu$.

\begin{re}
If $\MM$ is a group, then $\R\lbb\MM\rbb$ is a field.  In particular,
$\T_N = \R\lbb\G_N\rbb$ is a field ($N=0,1,2,\cdots$).
\end{re}

\begin{pr}\label{Elogfreepower}
Let $A$ be a nonzero log-free transseries.  If $A \fgt 1$, then
there exists a real number $c>0$ such that $A \fgt x^c$.
If $A \fst 1$, then there exists a real number $c<0$
such that $A \fst x^c$.
\end{pr}
\begin{proof}
Let $\mag A = x^b e^L \fgt 1$.
If $L=0$, then $b>0$, so take $c=b/2$.  If $L > 0$,
$A \fgt x^1$, since $\fgt$ is defined lexicographically.
The other case is similar.
\end{proof}

\begin{pr}[Height Wins]\label{Eheightwins}
Let $L>0$ be purely large of height $N$ and not $N-1$, let
$b \in \R$, and let
$T\ne 0$ be of height $N$.  Then $x^b e^L \fgt T$
and $x^b e^{-L} \fst T$.
\end{pr}
\begin{proof}
By induction on the height.
Let $\mag T = x^{b_1} e^{L_1}$.  So $L_1 \in \T_{N-1}$,
and therefore by the induction hypothesis $\dom(L-L_1) = \dom(L) > 0$.
So $L > L_1$ and $x^b e^L \fgt x^{b_1} e^{L_1}$.
\end{proof}

If $\n \in \G_N\setminus \G_{N-1}$ (we say $\n$ has \Def{exact height} $N$),
then either (i)~$\n \fgt 1$ and $\n \fgt \m$ for all $\m \in \G_{N-1}$,
or (ii)~$\n \fst 1$ and $\n \fst \m$ for all $\m \in \G_{N-1}$.
[We say $\G_{N-1}$ is \Def{convex} in $\G_N$.]

\begin{com}
$\n = e^{-e^x}$ has exact height $2$, and
$T = \sum_{k=1}^\infty \sum_{j=1}^\infty x^{-j} e^{-kx}$ has height $1$, so
of course $\n \fst T$.  Even more: $T/\n$ is purely large.
\end{com}

\subsection*{Derivative}

\begin{de}
\Def{Derivative} (notations ${}'$, $\partial$) is defined recursively.
First, $(x^a)' = a x^{a-1}$.  (If we are keeping track
of generating ratios, we may need the addendum
of ratio $x^{-1}$.)
If $\partial$ has been defined for $\G_N$,
then define it termwise for $\T_N$:
\begin{equation*}
	\left(\sum a_{\g} \g\right)' = \sum a_{\g} \g' .
\end{equation*}
(See the next proposition
for the proof that this makes sense.)
Then, if $\partial$ has been defined for $\T_N$,
define it on $\G_{N+1}$ by
\begin{equation*}
	\left(x^b e^L\right)' =
	b x^{b-1} e^L + x^{b} L' e^L 
	= \big(bx^{-1}+L'\big) x^b e^L .
\end{equation*}
\end{de}

For the \Def{derivative addendum} $\tbmu$:
begin with $\bmu$, add the heredity addendum of $\bmu$,
and add $x^{-1}$.  So (by induction) if
$T \in \T^{\ebmu}$, then $T' \in \T^{\tebmu}$.  Repeating
the derivative addendum adds nothing new, so in fact
all derivatives $T^{(j)}$ belong to $\T^{\tebmu}$.

\begin{re}
This derivative satisfies all the usual algebraic properties
of the derivative.  There are just lots of tedious things to
check.  $(AB)' = A' B + A B'$, $(S^k)' = k S' S^{k-1}$, etc.
\end{re}

\begin{pr}\label{Ederivexist}
Let $\bmu$ be given.  Let $\tbmu$ be as described.
{\rm (i)} If $T_i \muto T$ then 
$T_i' \overset{\tebmu}{\longrightarrow} T'$.
{\rm (ii)} If $\sum T_i$ is $\bmu$-convergent, then
$\sum T'_i$ is $\tbmu$-convergent and
$\big(\sum T_i\big)' = \sum T'_i$.
{\rm (iii)} If $\AA \subseteq \GRID^{\ebmu,\bm}$, then
$\sum_{\g \in \AA} a_\g \g'$ is $\tbmu$-convergent.
\end{pr}
\begin{proof}
(iii) is stated equivalently: the family $(\supp \g')$ is
point-finite.  Or: as $\g$ ranges over $\GRID^{\ebmu,\bm}$,
we have $\g' \overset{\tebmu}{\longrightarrow} 0$.

Proof by induction on the height.

Say $\mu_1 = x^{-b_1}e^{-L_1}, \cdots, \mu_n = x^{-b_n}e^{-L_n}$,
and $\bk = (k_1,\cdots,k_n)$.
Then
\begin{align*}
	\big(\bmu^\bk\big)'
	&=  \left(x^{-k_1 b_1 - \cdots - k_n b_n}
	e^{-k_1 L_1 - \cdots - k_n L_n}\right)'
	\\ &=  (-k_1 b_1 - \cdots - k_n b_n) x^{-1}\bmu^\bk
	+ (-k_1 L'_1 - \cdots - k_n L'_n) \bmu^\bk .
\end{align*}
So if $T = \sum_{\bk \ge \bm} a_{\bk} \bmu^\bk$, then
summing the above transmonomial result, we get
\begin{equation*}
	T' = x^{-1} T_0 + L'_1 T_1 + \dots + L'_n T_n,
\end{equation*}
where $T_0,\cdots, T_n$ are transseries with the same
support as $T$, and therefore they exist in $\T^{\ebmu,\bm}$.
Derivatives $L'_1, \cdots, L'_n$ exist by induction hypothesis.
So $T'$ exists.
\end{proof}

The preceding proof suggest the following.  Think of
$\lsupp$ as ``the support of the logarithmic derivative''
for monomials.

\begin{de}\label{Elsupp}
For (log-free)
monomials, define
\begin{equation*}
	\lsupp(x^b e^L) = \{x^{-1}\} \cup \supp L' .
\end{equation*}
For a set $\AA \subseteq \G_{\bullet}$, define
$\lsupp \AA = \bigcup_{\g \in E} \lsupp \g$.
For $T \in \T_{\bullet}$, define
$\lsupp T = \lsupp\supp T$.  For a set $\A \subseteq \T_{\bullet}$,
define $\lsupp \A = \bigcup_{T \in \A} \lsupp T$.
\end{de}

\begin{pr} Properties of $\lsupp$.
\begin{enumerate}
\item[{\rm (a)}] If $\bmu = \{\mu_1,\cdots,\mu_n\} \subseteq \Gsmall_{\bullet}$
and $\bk \in \Z^n$, then $\lsupp \bmu^\bk \subseteq \lsupp \bmu$.
So $\lsupp \T^{\ebmu} \subseteq \lsupp \bmu$.
\item[{\rm (b)}] For any finite $\bmu \subseteq \Gsmall_{\bullet}$, $\lsupp(\bmu)$
is a subgrid.
\item[{\rm (c)}] If $T \in \T^{\ebmu}$, then
$\supp(T') \subseteq \lsupp(\bmu)\cdot \supp(T)$.
\item[{\rm (d)}] If $\tbmu$ is the smallness addendum
for some $S \in \T^{\ebmu}$, then
$\lsupp \widetilde \bmu \subseteq \lsupp \bmu$.
\item[{\rm (e)}] If $\g \in \G_N$, $N \ge 1$, then
$\lsupp \g \subseteq \G_{N-1}$.
\item[{\rm (f)}] If $\g \in \G_0$, $\g \ne 1$, then
$\lsupp \g = \{x^{-1}\}$.
\item[{\rm (g)}] Each $\T_N$ is closed under $\partial$.
\end{enumerate}
\end{pr}

\begin{proof}
(a) Say $\mu_1 = x^{-b_1} e^{-L_1},\cdots,\mu_n = x^{-b_n} e^{-L_n}$.
Then
\begin{align*}
	\bmu^\bk &=  x^{-k_1 b_1 -\dots - k_n b_n}
	e^{-k_1 L_1 -\cdots -k_n L_n} ,\text{ so}
	\\
	\lsupp \bmu^\bk &=  \{x^{-1}\} \cup \supp(-k_1 L'_1 -\cdots -k_n L'_n)
	\\ &\subseteq
	\{x^{-1}\} \cup \supp(L'_1) \cup \cdots \cup \supp(L'_n) = \lsupp \bmu.
\end{align*}

(b) Each $\supp L'_i \in \W$ and $\{x^{-1}\} \in \W$, so their (finite)
union also belongs to $\W$.

(c) Proposition~\ref{Ederivexist}.

(d) Use the proof of Proposition~\ref{addendumsmall} together with (a).

(e) and (f) are clear.

(g) Use (c) and (e).
\end{proof}

Similar to (d):
If $\tbmu$ is the inversion addendum
for some $A \in \T^{\ebmu}$, then
$\lsupp \widetilde \bmu \subseteq \lsupp \bmu$.

\begin{re}
Note $\partial$ maps $\G_N$ into $\T_N$, so $\T_N$ is
closed under $\partial$.
\end{re}

Here are a few technical results on derivatives.
They lead to Proposition~\ref{constderiv},
where we will prove that
$T'=0$ only if $T$ is ``constant''
in the sense used here, that $T \in \R$.

\begin{pr}\label{nox^-1}
There is no $T \in \T_{\bullet}$ with $T' = x^{-1}$.
\end{pr}
\begin{proof}
In fact, we show:  If $\g \in \G_{\bullet}$, then
$x^{-1} \not\in \supp \g'$.  This suffices since
\begin{equation*}
	\supp T' \subseteq \bigcup_{\g \in \supp T} \supp \g' .
\end{equation*}
Proof by induction on the height.  If $\g = x^b$ has height $0$,
then $\g' = b x^{b-1}$ and $x^{-1} \not\in \supp \g'$.
If $\g = x^b e^L$ has exact height $N$, so
$L$ is purely large of exact height $N-1$, then
$\g' = (bx^{-1}+L') x^be^L$.  Now by the induction hypothesis,
$bx^{-1}+L' \ne 0$, so (by Proposition~\ref{Eheightwins})
every term of $\g'$ is far larger than $x^{-1}$ if $L>0$ and
far smaller than $x^{-1}$ if $L<0$.  So $x^{-1} \not\in \supp\g'$.
\end{proof}

\begin{pr}\label{derivprops}
{\rm (a)} Let $\m \ne 1$ be a log-free monomial with exact
height $n$.  Then $\m'$ also has exact height $n$.
{\rm (b)} If $\m \fgt \n$, and $\m \ne 1$, then
$\m' \fgt \n'$.
{\rm (c)}~If $\mag T \ne 1$, then $T' \cto (\mag T)'$ and
$T' \ato (\dom T)'$.
{\rm (d)}~If $\mag T \ne 1$ and $T \fgt S$, then $T' \fgt S'$.
\end{pr}
\begin{proof} (a) For height $0$, $\m = x^b, b\ne0$ so
$\m' = bx^{b-1} \ne 0$ also has height $0$.
Let $\m = x^be^L \ne 1$ have exact height
$n$, so that $L\ne 0$ has exact height $n-1$.  Of course $L'$
has height at most $n-1$, and $(bx^{-1}+L')$ is not zero
by Proposition~\ref{nox^-1},
so $\m' = (bx^{-1}+L')x^be^L$ again has exact
height $n$.

If (b) holds for all $\m, \n$ of a given height $n$, then
(c) and (d) follow for $S, T$ of height $n$.
So it remains to prove (b).

First suppose $\m,\n$ have different heights.  Say
$\m \in \G_m, \n \in \G_n\setminus\G_m, n>m$.
If $\n \fgt 1$, then $\n' \fgt 1$ (because its height $n>0$),
$\n' \in \G_n\setminus\G_m$, $\m \fst \n$
by Proposition~\ref{Eheightwins}, $\m' \fst \n'$. 
If $\n \fst 1$, then $\n' \fst 1$,
$\n' \in \G_n\setminus\G_m$, $\m \fgt \n$, $\m' \fgt \n'$.
So (b) holds in both cases.

So suppose $\m,\n$ have the same height $n$.
Write $\m = x^a e^A, \n = x^b e^B$, where $a,b$ real and
$A,B$ purely large.  Assume $\m \fgt \n$, so either
$A>B$ or $A=B, a>b$.  We take the case $A>B$ (the other one
is similar to Case 2, below).  Say $A-B$ has
exact height $k$.  There will be two cases: $k = n-1$ and $k < n-1$.

\textit{Case 1.\/}  $k=n-1$.  Then $A-B$ has exact height $n-1$ and
\begin{equation*}
	\frac{bx^{-1}+B'}{a x^{-1} +A'}
\end{equation*}
has height $n-1$ (and its denominator is not zero
by Proposition~\ref{nox^-1}).
Therefore by Proposition~\ref{Eheightwins},
\begin{equation*}
	x^{a-b}e^{A-B}\fgt\frac{bx^{-1}+B'}{a x^{-1} +A'},
\end{equation*}
and thus
$(a x^{-1} +A')x^a e^A \fgt (bx^{-1}+B')x^b e^B$.
That is, $\m' \fgt \n'$.

\textit{Case 2.\/} $k<n-1$.  Write
$A = A_0+A_1, B = B_0+A_1$ where purely large
$A_0, B_0$ have height $k$
(and purely large $A_1 \ne 0$ has exact height $n-1 > k$).
Now $A_1'$ has height
$n-1 > k$ and is large, so 
$A_1' \fgt ax^{-1}+A_0'$ and $A_1' \fgt bx^{-1}+B_0'$.  Since
$x^{a-b}e^{A_0-B_0} \fgt 1$, we have
$x^a e^{A_0} \fgt x^b e^{B_0}$ and therefore
\begin{align*}
	\m' &= 
	(ax^{-1}+A_0'+A_1')x^a e^{A_0+A_1}
	\ato
	A_1' x^a e^{A_0+A_1}
	\\ & \qquad\fgt
	A_1' x^b e^{B_0+A_1}
	\ato
	(bx^{-1}+B_0'+A_1') x^b e^{B_0+A_1}
	= \n' .
\end{align*}
(See \cite[Prop.~4.1]{DMM} for this proof.)
\end{proof}

\begin{pr}\label{EWKB}
Let $T \in \T_{\bullet}$.
{\rm (i)} If $T \fst 1$, then $T' \fst 1$.
{\rm (ii)}~If $T \fgt 1$ and $T>0$, then $T'>0$.
{\rm (iii)}~If $T \fgt 1$ and $T<0$, then $T'<0$.
{\rm (iv)}~If $T \fgt 1$, then $T^2 \fgt T'$.
{\rm (v)}~If $T \fgt 1$, then $(T')^2 \fgt T''$.
{\rm (vi)}~If $T \fgt 1$ then $xT' \fgt 1$.
\end{pr}
\begin{proof}
(i) $T \fst 1 \quad\Longrightarrow\quad T \fst x
\quad\Longrightarrow\quad T' \fst 1$.

(ii)(iii) Assume $T \fgt 1$.  Let $\dom T = a x^b e^L$, so
$T$ has the same sign as $a$. Then $T' \sim a (bx^{-1}+L')x^b e^L$.
The proof is by induction on the height of $T$.  If $T$ has height $0$,
so that $L=0$ and $b>0$, then $T'\sim abx^{b-1}e^L$ has the same
sign as $a$.  Assume $T$ has height $N > 0$, so $L > 0$ and $L$
has height $N-1$, so (since $L$ is large)
the induction hypothesis tells us that $L' > 0$.
Also, $L \fgt x^c$ for some $c>0$ so $L' \fgt x^{c-1} \fgt x^{-1}$,
so $T' \sim abL'x^be^L$ has the same sign as $a$.

(iv)  $\displaystyle
T \fgt 1 \;\;\Longrightarrow\;\; T \fgt \frac{1}{x}
\;\;\Longrightarrow\;\; \frac{1}{T} \fst x
\;\;\Longrightarrow\;\; \left(\frac{1}{T}\right)' \fst 1
\;\;\Longrightarrow\;\; \frac{T'}{T^2} \fst 1$.

(v) $T \fgt x^c$ for some $c>0$, so $T' \fgt x^{c-1} \fgt x^{-1}$,
then proceed as in (iv).

(vi) $T' \fgt 1/x$ as in the proof of (v).
\end{proof}

\begin{com}
After we define real powers (Definition~\ref{realpowers}),
we will be able to formulate
a variant of (iv):
If $T\fgte 1$, then $|T|^{1+\epsilon} \fgt T'$
for all real $\epsilon>0$;
if $T \fst 1$, then $|T|^{1-\epsilon} \fgt T'$
for all real $\epsilon>0$
\cite[Prop.~4.1]{DMM}.
And if $T > x^a$ for all real $a$, then
$T^{1-\epsilon} \fst T' \fst T^{1+\epsilon}$
for all real $\epsilon>0$
\cite[Cor.~4.4]{DMM}.
The proofs are essentially as given above for (iv).
\end{com}

\begin{pr}
{\rm (a)} If $L \ne 0$ is large and $b \in \R$, then
$\dom\big((x^be^L)'\big) = x^be^L\dom(L')$.
{\rm (b)}~If $\g \in \G_{\bullet}$, $\g \ne 1$, then $\g' \ne 0$.
\end{pr}
\begin{proof}
(a) Since $L \fgt 1$, there is $c>0$ with $L \fgt x^c$,
so $L' \fgt x^{c-1} \fgt x^{-1}$.  So
$(x^be^L)' = x^be^L(b x^{-1} +L') \cto x^be^L L'$.
For (b), use induction on the height and (a).
\end{proof}

\begin{pr}\label{constderiv}
Let $T \in \T_{\bullet}$.  If $T'=0$, then $T$ is a constant.
\end{pr}
\begin{proof}
Assume $T'=0$.
Write $T = L + c + S$ as in \ref{C_add}.  If $L \ne 0$ then 
$T' \cto (\mag L)' \ne 0$
so $T' \ne 0$.
If $L=0$ and $S\ne 0$, then $T' \cto (\mag S)' \ne 0$ so $T' \ne 0$.
Therefore $T=c$.
\end{proof}

The set $\T_N$ is a \Def{differential field} with constants $\R$.
This means it follows the rules you already know for computations
involving derivatives.

\begin{pr}[Addendum Height]\label{addendumheight}
Let $\bmu \subseteq \Gsmall_N$, and let $T \in \T^{\ebmu}$.
{\rm (i)}~If $\tbmu$ is the smallness addendum
for $T$, then $\tbmu \subseteq \Gsmall_N$.
{\rm (ii)}~If $\tbmu$ is the inversion addendum
for $T$, then $\tbmu \subseteq \Gsmall_N$.
{\rm (iii)}~If $\tbmu$ is the heredity addendum
for $\bmu$, then $\tbmu \subseteq \Gsmall_N$.
{\rm (iv)}~If $\tbmu$ is the derivitive addendum
for $T$, then $\tbmu \subseteq \Gsmall_N$.
\end{pr}

\subsection*{Compositions}
The field of transseries has an operation of ``composition.''
The result $T \circ S$ is, however, defined in
general only for some $S$.  We will
start with the easy cases.

\begin{de}\label{realpowers}
We define $S^b$, where $S \in \T_{\bullet}$ is positive,
and $b \in \R$.  First, write $S = cx^ae^L(1+U)$ as
in \ref{C_mult}, with $c>0$.  Then define
$S^b = c^b x^{ab} e^{bL}(1+U)^b$.  Constant $c^b$, with $c>0$,
is computed in $\R$.  Next, $x^{ab}$ is a transmonomial,
but (if we are keeping track of generating ratios)
may require addendum of a ratio.  Also,
$(1+U)^b$ is a convergent binomial series, again we may require
the smallness addendum for $U$.  Finally, since $L$ is purely
large, so is $bL$, and thus $e^{bL}$ is a transmonomial, but
may require addendum of a ratio.
\end{de}

\begin{re}
Note $S^b$ is not of greater height than $S$: If
$S \in \T_N$, then $S^b \in \T_N$.  If $b \ne 0$,
then because $(S^b)^{1/b} = S$, in fact the exact
height of $S$ is the same as $S^b$.
\end{re}

\begin{com}
Monotonicity: If $b>0$ and $S_1 < S_2$, then
$S_1^b < S_2^b$.  If $b<0$ and $S_1 < S_2$, then
$S_1^b > S_2^b$.
\end{com}

\begin{de}
We define $e^S$, where $S \in \T_{\bullet}$.
Write $S = L + c + U$ as in \ref{C_add}, with $L$ purely large,
$c$ a constant, and $U$ small.  Then $e^S = e^L e^c e^U$.  Constant
$e^c$ is computed in $\R$.  [Note that $e^S>0$ since the
leading coefficient is $e^c$.]
Next, $e^U$ is a power series (with point-finite convergence); we may need
the smallness addendum for $U$.  And of course $e^L$ is a transmonomial,
but might not already be a generating ratio, so perhaps $e^L$ or $e^{-L}$ is
required as addendum.
\end{de}

\begin{re}
Of course, if $S=L$ is purely large, then this definition
of $e^S$ agrees with the formal notation $e^L$ used before.
Height increases by at most one:
If $S \in \T_N$, then $e^S \in \T_{N+1}$.
\end{re}

\begin{com}
Monotonicity:  If $S_1 < S_2$ then
$e^{S_1} < e^{S_2}$.
\end{com}

\begin{de}\label{expcomp}
Let $S, T \in \T_{\bullet}$ with $S$ positive and large (but not necessarily
purely large).  We want to define the \Def{composition} $T \circ S$.
This is done by induction on the height of $T$.
When $T = x^b e^L$ is a transmonomial, define
$T \circ S = S^b \, e^{L \circ S}$.
Both $S^b$ and $e^{L \circ S}$
may require addenda.  And $L \circ S$ exists by
the induction hypothesis.  In general, when
$T = \sum c_\g \g$, define $T \circ S = \sum c_\g (\g \circ S)$.
The next proposition is required.
\end{de}

\begin{re}
If $T \fgt 1$, then $T \circ S \fgt 1$.  If $T \fst 1$,
then $T \circ S \fst 1$.
Because of our use of the
symbol $x$, it will not be unexpected if we sometimes write
$T(S)$ for $T \circ S$.
Alternate term: ``large and positive'' = ``infinitely increasing''.
\end{re}

\begin{pr}\label{Ecompexist}
Let $\GRID^{\ebmu,\bm}$ be a log-free grid and let $S \in \T_{\bullet}$ be a large, positive, log-free transseries.
Then there exist
$\tbmu$ and $\widetilde{\bm}$ so that
$\g \circ S \in \T^{\tebmu,\tilde{\bm}}$
for all $\g \in \GRID^{\ebmu,\bm}$,
and the family
$\big(\supp(\g \circ S)\big)$ is point-finite.
\end{pr}
\begin{proof}
First, add the heredity addendum of $\bmu$.
Now for the ratios $\mu_1,\cdots,\mu_{m}$,
write $\mu_i = x^{-b_i} e^{-L_i}$, $1 \le i \le m$.
Arrange the list so that for all~$i$,
$L_i \in \T^{\{\mu_1,\cdots,\mu_{i-1}\}}$.
Then take the $\mu_i$ in order.
Each $S^{-b_i}$ may require an addendum.  
Each $L_i \circ S$ may require an addendum.
So all $\mu_i \circ S$
exist.  They are small.  Add smallness addenda for
these.  So finally we get $\tbmu$.

Now for each $\mu_i \in \bmu$, we have $\mu_i \circ S$
is $\tbmu$-small. So by Proposition~\ref{finiteseries}
we have $(\g \circ S)_{\g \in \GRID^{\ebmu,\bm}}
\overset{\tebmu}{\longrightarrow} 0$.
\end{proof}

\begin{com}
Note $\tbmu$ depends on $S$, not just on a ratio set
generating $S$.

For composition $T \circ S$, we need $S$ to be large.
Example: Let $T = \sum_{j=0}^\infty x^{-j}$, 
$S = x^{-1}$.  Then $S$ is small, not large.  And
$T \circ S = \sum_{j=0}^\infty x^j$ is not
a valid transseries.
\end{com}

\begin{re}
If $S \in \T_{N_1}$ and $T \in \T_{N_2}$,
then $T \circ S \in \T_{N_1+N_2}$.
\end{re}

The following is proved using the recursive definition.
Note that $\G_n \cup \Gsmall_N$ is a \Def{convex subset}
of $\G_\bullet$.

\begin{pr}\label{compo_nN}
Let $n,N \in \N$ with $n \le N$.  Assume
$T,B \in \T_\bullet$, $\supp T \subset \G_n \cup \Gsmall_N$,
$\supp B \subseteq \G_N$, $B \fst x$.
Then $\supp\big(T\circ(x+B)\big) \subset \G_n\cup\Gsmall_N$.
\end{pr}

\subsection*{Continuity of Composition}

\begin{pr}\label{Ecompcontin}
Let $S$ be large and positive.  Let $(T_i)$ be a family
of transseries with $T_i \to T$.  Then
$T_i \circ S \to T \circ S$.
\end{pr}
\begin{proof}
Say $T_i \overset{\ebmu,\bm}{\longrightarrow} T$.
Let $\tbmu$ and $\widetilde{\bm}$ be
as in Proposition~\ref{Ecompexist} so that
$\g\circ S \in \T^{\tebmu,\tilde{\bm}}$ for all
$\g \in \GRID^{\ebmu,\bm}$, and
$\g\circ S \overset{\tebmu}{\longrightarrow} 0$.
Let $\m \in \G_{\bullet}$.  There are finitely many
$\g \in \GRID^{\ebmu,\bm}$ such that $\m \in \supp (\g \circ S)$.
For each such $\g$ there are finitely many $i$ such that
$\g \in \supp (T_i-T)$.  So if $i$ is outside this finite union
of finite sets, we have $\m \not\in \supp((T_i-T) \circ S)$.
\end{proof}

\begin{com}
Continuity in the other composand might not hold.  For
$j \in \N$, let $S_j = x^{-j}e^{x}$.
Then $(S_j) \to 0$.  But the family
$(\exp(S_j))$ is not supported by any grid,
so $\exp(S_j)$ cannot converge to anything.
\end{com}

Tedious calculation should show that the usual derivative formulas hold:
$(S^b)' = bS^{b-1} S'$, $(e^T)' = e^T T'$,
$(T\circ S)' = (T' \circ S)\cdot S'$, and so on.

\section*{3E\quad With Logarithms}
\addcontentsline{toc}{section}{3E\quad With Logarithms}%
Transseries with logs are obtained by formally
composing the log-free transseries with $\log$ on the right.

\begin{no}
If $m \in \N$, we write $\log_m$ to represent the $m$-fold
composition of the natural logarithm with itself;  $\log_0$ will
have no effect; sometimes we may write $\log_{m} = \exp_{-m}$,
especially when $m<0$.
\end{no}

\begin{de}
Let $M \in \N$.  A transseries with \Def{depth} $M$ is a formal
expression $Q = T \circ \log_M$, where $T \in \T_{\bullet}$.
\end{de}

We identify the set of transseries of depth $M$ as a subset of the
set of transseries
of depth $M+1$ by identifying $T \circ \log_M$
with $(T\circ\exp)\circ \log_{M+1}$.  Composition on
the right with $\exp$ is defined in Definition~\ref{expcomp}.
Using this idea, we define operations on transseries
from the operations in $\T_{\bullet}$.

\begin{de}
Let $Q_j = T_j \circ \log_M$,
where $T_j \in \T_{\bullet}$.  Define
$Q_1 + Q_2 = (T_1 + T_2)\circ \log_M$;
$Q_1 Q_2 = (T_1 T_2)\circ \log_M$;
$Q_1 > Q_2$ iff $T_1 > T_2$;
$Q_1 \fgt Q_2$ iff $T_1 \fgt T_2$;
$Q_j \to Q_0$ iff $T_j \to T_0$;
$\sum Q_j = (\sum T_j) \circ \log_M$;
$Q_1^b = (T_1^b)\circ \log_M$;
$\exp(Q_1) = (\exp(T_1))\circ \log_M$;
and so on.
\end{de}

\begin{de} \Def{Transseries.}  Always assumed grid-based.
{\allowdisplaybreaks
\begin{align*}
	\G_{NM} &=  \SET{\g \circ \log_M}{\g \in \G_N} ,
	\\
	\T_{NM} &=  \SET{T \circ \log_M}{T \in \T_N} = \R\lbb\G_{NM}\rbb,
	\\
	\G_{\bullet M} &=  \bigcup_{N \in \N} \G_{NM}
	= \SET{\g \circ \log_M}{\g \in \G_{\bullet}} ,
	\qquad \G_{N\bullet} = \bigcup_{M \in \N} \G_{NM} ,
	\\
	\T_{\bullet M} &=  \bigcup_{N \in \N} \T_{NM}
	= \SET{T \circ \log_M}{T \in \T_{\bullet}} =\R\lbb\G_{\bullet M}\rbb,
	\quad \T_{N\bullet} = \R\lbb\G_{N\bullet}\rbb ,
	\\
	\G_{\bullet\bullet} &=  \bigcup_{M \in \N} \G_{\bullet M} 
	= \bigcup_{N,M \in \N}\G_{NM},
	\\
	 \R\lbbb x \rbbb &=  \bigcup_{M \in \N} \T_{\bullet M} = 
	 \bigcup_{N,M \in \N}\T_{NM} =
	\R\lbb\G_{\bullet\bullet}\rbb = \T_{\bullet\bullet}.
\end{align*}
}%
When $M<0$ we also write $\T_{\bullet M}$.  So for example
$\T_{\bullet,-1} = \SET{T \circ \exp}{T \in \T_{\bullet}}$.
\end{de}

If $T = \sum c_\g \g$ we may write $T \circ \log_M$ as a series
\begin{equation*}
	\left(\sum c_\g \g\right) \circ \log_M = \sum c_\g (\g \circ \log_M) .
\end{equation*}
Simplifications along these lines may be carried out:
$\exp(\log x) = x$;
$e^{b \log x} = x^b$; etc.
As usual we sometimes think of $x$ as a variable and
sometimes as the identity function.
On monomials we can write
\begin{equation*}
	(x^b e^L)\circ \log = (\log x)^b e^{L\circ\log} ,
\end{equation*}
and continue recursively in the exponent

We say
$Q \in \R\lbbb x \rbbb$ has \Def{exact depth} $M$ iff $Q = T \circ \log_M$,
$T \in \T_{\bullet}$
and $T$ cannot be written in the form $T = T_1 \circ \exp$
for $T_1 \in \T_{\bullet}$.  This will also make sense for
negative $M$.

\subsection*{Terminology}

The group $\G = \G_{\bullet\bullet}$ is the group of \Def{transmonomials}.
The ordered field $\T = \T_{\bullet\bullet}=\R\lbbb x \rbbb$
is the field of
\Def{transseries}.  Van der Hoeven \cite{hoeven} calls $\T$
\Def{the transline}.

\begin{com}
Although $x^x$ is not an ``official'' transmonomial,
if we consider it to be
an abbreviation for $e^{x\log x}$, then it may be considered
to be a transmonomial according to our identifications:
\begin{equation*}
	x^x = e^{x\log x} = \left(e^{e^x x}\right)\circ \log .
\end{equation*}
So $x^x$ has height 2 and depth 1; that is, $x^x \in \G_{2,1}$.
\end{com}

\begin{com}
Just as we require finite exponential height, we also require
finite logarithmic depth.  So the following is \emph{not}
a grid-based transseries:
\begin{equation*}
	x + \log x + \log\log x + \log\log\log x + \cdots .
\end{equation*}
But see for example \cite{schmeling} or
\cite{edgarc} for a variant that allows this.
\end{com}

\begin{com}\label{e^L}
If $\g \in \G_{\bullet\bullet}$, then $\g = e^L$ for some purely large
$L \in \T_{\bullet\bullet}$.  Because of logarithms, there is no need
for an extra $x^b$ factor.
\end{com}

\begin{de}\label{de:log}
\Def{Logarithm}.  If $T \in \T_{\bullet}$, $T>0$, write
$T = a x^b e^L (1+S)$ as in \ref{C_mult}.  Define
$\log T = \log a + b\log x + L + \log(1+S)$.
Now $\log a$, $a>0$, is computed in $\R$.
And $\log(1+S)$ is computed as a Taylor series.
The term $b \log x$
gives this depth $1$; if $b=0$ then we remain log-free.

For general $Q \in \R\lbbb x \rbbb$: if $Q = T\circ \log_M$, then
$\log(Q) = \log(T) \circ \log_M$, which could have depth $M+1$.

Alternatively (from Comment~\ref{e^L}): for $Q \in \T_{\bullet\bullet}$, write
$Q = a e^L(1+S)$ and then $\log Q = L + \log a + \log(1+S)$.
\end{de}

\begin{com}
If $T \not\cto 1$, then $\log T \fgt 1$.
\end{com}

\begin{de}
\Def{Composition}.  Let $Q_1, Q_2 \in \R\lbbb x \rbbb$
with $Q_2$ large and positive.  Define $Q_1 \circ Q_2$
as follows:  Write $Q_1 = T_1 \circ \log_{M_1}$ and
$Q_2 = T_2 \circ \log_{M_2}$, with $T_1, T_2 \in \T_{\bullet}$.
Applying \ref{de:log} $M_1$ times, we can write
$\log_{M_1}(T_2) = S \circ \log_{M_1}$ with $S \in \T_{\bullet}$.
Then define:
\begin{equation*}
	Q_1 \circ Q_2 = T_1 \circ \log_{M_1} \circ
	T_2 \circ \log_{M_2} =
	T_1 \circ S \circ \log_{M_1+M_2} ,
\end{equation*}
and compute $T_1 \circ S$ as in \ref{expcomp}.
\end{de}

The set of large positive transseries from $\R\lbbb x \rbbb$ is closed
under composition.  In fact, it is a group \cite[p.~111]{hoeven}.

\begin{de}
Differentiation is done as expected from the usual rules.
\begin{equation*}
	\big(T \circ \log\big)' = (T' \circ \log) \cdot x^{-1}
	= \big( T' e^{-x}\big) \circ\log .
\end{equation*}
So $\partial$ maps $\T_{\bullet M}$ into itself.
\end{de}

\begin{com}
\dots but perhaps $\partial$ does not map $\T_{NM}$ into itself.
Example: $Q = (\log x)^2 = (x^2) \circ \log \in \T_{0,1}$,
and $Q' = 2(\log x)/x = (2xe^{-x})\circ \log \in \T_{1,1}$
but $Q' \not\in \T_{0,1}$.  See
Remark~\ref{derivlogm}.
\end{com}


We now have an antiderivative for $x^{-1}$.
\begin{equation*}
	\big(\log x\big)' = \big(x \circ \log\big)'
	= \big(1 \cdot e^{-x}\big)\circ \log
	= (x^{-1})\circ\exp\circ\log = x^{-1} .
\end{equation*}
We will see below (Proposition~\ref{Eintegral}) that,
in fact, every transseries has
an antiderivative.

Here are some simple properties of the derivative.

\begin{pr}\label{derivproperties}
Let $S,T \in \R\lbbb x \rbbb$, $\n,\m \in \G_{\bullet\bullet}$.
{\rm (a)}~If $\m \fgt \n$, and $\m \ne 1$, then
$\m' \fgt \n'$.
{\rm (b)}~If $\mag T \ne 1$, then $T' \cto (\mag T)'$ and
$T' \ato (\dom T)'$.
{\rm (c)}~If $\mag T \ne 1$ and $T \fgt S$, then $T' \fgt S'$.
{\rm (d)}~If $T \fst 1$, then $T' \fst 1$.
{\rm (e)}~If $T \fgt 1$ and $T>0$, then $T'>0$.
{\rm (f)}~If $T \fgt 1$ and $T<0$, then $T'<0$.
{\rm (g)}~If $T \fgt 1$, then $T^2 \fgt T'$.
{\rm (h)}~If $T \fgt \log x$, then $(T')^2 \fgt T''$.
{\rm (i)}~If $T \fgt \log x$, then $xT' \fgt 1$.
{\rm (j)}~If $T'=0$, then $T$ is a constant.
\end{pr}

\begin{proof}
(a)(b)(c) Starting with Proposition~\ref{derivprops}(b)(c)(d), compose
with $\log$ repeatedly.

(d) $T \fst 1 \quad\Longrightarrow\quad T \fst x
\quad\Longrightarrow\quad T' \fst 1$.

(e)(f) Starting with Proposition~\ref{EWKB}(ii)(iii),  compose
with $\log$ repeatedly.

(g) $\displaystyle
T \fgt 1 \;\;\Longrightarrow\;\; T \fgt \frac{1}{x}
\;\;\Longrightarrow\;\; \frac{1}{T} \fst x
\;\;\Longrightarrow\;\; \left(\frac{1}{T}\right)' \fst 1
\;\;\Longrightarrow\;\; \frac{T'}{T^2} \fst 1$.

(h) since $T \fgt \log x$, we have $T' \fgt x^{-1}$,
then proceed as in (g).

(i) $T \fgt \log x \quad\Longrightarrow\quad T' \fgt x^{-1}$.

(j) Starting with Proposition~\ref{constderiv},  compose
with $\log$ repeatedly.
\end{proof}

\begin{com}
Note $T=\log\log x$ is a counterexample to:
If $T \fgt 1$ then $xT' \fgt 1$. And to:
If $T \fgt 1$, then $(T')^2 \fgt T''$.
\end{com}

\begin{re}\label{derivlogm}
Of course $\G_{NM}$ is a group, so $\T_{NM}$ is a field.
The derivative of $\log_M$ is
\begin{equation*}
	\big(\log_M x\big)' = \left(\prod_{m=0}^{M-1} \log_m x\right)^{-1} .
\end{equation*}
If $N\ge M$ then this belongs to $\G_{NM}$, so in that case
$\T_{NM}$ is a differential field (with constants $\R$).
\end{re}

\begin{pr}
If $T \in \T_{N,M}$ has exact height $N \ge 1$ {\rm (}that is,
$T \not\in \T_{N-1,\bullet}${\rm )} and $T' \fgt 1$,
then $T' \in \T_{N,M}$ and also has exact height $N$.
\end{pr}
\begin{proof}
There is $T_1 \in \T_N$ with $T(x) = T_1(\log_M x)$.
Then
\begin{equation*}
	T'(x) = \frac{T_1'(\log_M x)}{x\log x \cdots \log_{M-1} x} ,
\end{equation*}
so $T_1'(\log_M x) \fgt x$ and $T_1(x) \fgt \exp_M x$.
So $N \ge M$ and $x\log x \cdots \log_{M-1} x \in \T_{N-1,M}$
so $T' \in \T_{N,M}$.  Since $T_1'$ has exact height $N$ and
$x\log x \cdots \log_{M-1} x$ has height $\le N-1$, it follows
that $T'$ has exact height $N$.
\end{proof}

\subsection*{Valuation}
The map ``$\mag$'' from $\R\lbb\G\rbb \setminus \{0\}$ to $\G$ is a
(nonarchimedean) \Def{valuation}.  This means:
\begin{enumerate}
\item[(i)] $\mag(ST) = \mag(S) \mag(T)$;
\item[(ii)] $\mag(S+T) \fste \max\{\mag(S), \mag(T)\}$
with equality if $\mag(S) \ne \mag(T)$.
\end{enumerate}
The ordered group $\G$ is the \Def{valuation group}.

\begin{com}
The valuation group is
written multiplicatively here, but in many parts
of mathematics it is more common
to write it additively, and with the order reversed.
In the transseries case, $\G = \G_{\bullet\bullet}$,
we could follow this ``additive'' convention by saying:
the \Def{valuation group} is the set of purely large transseries,
with operation $+$ and order $<$.  
The valuation $\nu$ is then related to the magnitude
by: $\mag T = e^{-L} \Longleftrightarrow \nu(T) = L$.  We could then
still call $\G$ the
\Def{monomial group}.  But for a general ordered abelian
monomial group $\G$ (without $\log$ and $\exp$)
the valuation group would have to consist
of ``formal logarithms'' of the monomials;
introducing them may seem artificial.
\end{com}

The map ``$\mag$'' is an \Def{ordered valuation}.
This means that it also satisfies:
\begin{enumerate}
\item[(iii)] if $\mag(T) \fgt 1$ then $|T| > 1$. [The absolute value
$|T|$ is defined as usual.]
\end{enumerate}
The map ``$\mag$'' is a \Def{differential valuation}.
This means that it also satisfies:
\begin{enumerate}
\item[(iv)] if $\mag(T) \ne 1, \mag(S) \ne 1$, then
$\mag(T) \fste \mag(S)$ if and only if $\mag(T') \fste \mag(S')$;
\item[(v)] if $\mag(T) \fst \mag(S) \ne 1$,
then $\mag(T') \fst \mag(S')$.
\end{enumerate}
For more on valuations, see \cite{kuhlmann} and \cite{rosenlichttrans}.

\section{Example Computations}\label{excomp}
I will show here some computations.  They can be
done by hand with patience, but modern computer algebra systems
will handle them easily.  Read these or---better yet---try
doing some computations of your own.  I think that
your own experience with it will convince you better
than anything else that this system is truly elementary,
but very powerful.

\subsection*{A Polynomial Equation}
\begin{problm}\label{fifth}
Solve the fifth-degree polynomial equation
\begin{equation*}
	P(Y) := Y^5+e^x Y^2-x Y - 9 = 0
\end{equation*}
for $Y$.
\end{problm}

We can think of
this problem in various ways.  If $x$ is a real number,
then we want to solve for a real number $Y$.
(When $x=0$, the Galois group is $S_5$, so we will not be
solving this by radicals!)  Or: think of $x$ and $e^x$
as functions, then the solution $Y$ is to be a function
as well.  Or: think of $x$ and $e^x$ as transseries,
then the solution $Y$ is to be a transseries.
From some points of view, this last one is the easiest
of the three.  That is what we will do now.

In fact, many polynomial equations with
transseries coefficients have transseries
solutions.  Of course for solutions in
$\R\lbbb x \rbbb$ there are certain restrictions,
since some polynomials (such as $Y^2+1$) might have no zeros
because they are always positive.  But if
$P(Y) \in \R\lbbb x \rbbb [Y]$ and there are
transseries $A, B$ with $P(A) > 0, P(B) < 0$,
there there is a transseries $T$ between
$A$ and $B$ with
$P(T) = 0$.  Van der Hoeven  \cite[Chap.~9]{hoeven} has
this is even for \emph{differential} polynomials.

The transseries coefficients of our equation
$Y^5+e^x Y^2-x Y - 9 = 0$
belong to the set $\T_1 = \T_{1,0}$ of height $1$ depth $0$ transseries.
Our solutions will also be in $\T_1$.
Note that $Y=0$ is not a solution.  So any solution $Y$ has
a dominance $\dom Y = a x^b e^L$, where $a \ne 0$ and $b$
are real, and $L \in \T_0$ is purely large.  So the dominances
for the terms are:
{\allowdisplaybreaks
\begin{align*}
	\dom (Y^5) &=  a^5 x^{5b} e^{5L} ,\\
	\dom (e^x Y^2) &=  a^2 x^{2b} e^{2L+x} ,\\
	\dom (-x Y) &=  -a x^{b+1}e^L ,\\
	\dom (-9) &=  -9 .
\end{align*}
}%
Now we can compare these four terms.
If $L>x/3$, then $Y^5$ is far larger than
any other term, so $P(Y) \cto Y^5$ and therefore $P(Y) \ne 0$.
If $-x/2 < L < x/3$, then
$P(Y) \cto e^x Y^2$.
If $L < -x/2$, then $P(T) \cto -9$
(all other terms are ${} \fst 1$).
So the only possibilities for $L$ are $x/3$ and $-x/2$.
[If you know the ``Newton polygon'' method, you may recognize
what we just did.]

We consider first $L=x/3$.  If $b>0$ then $P(Y) \cto Y^5$;
if $b<0$ then $P(Y) \cto e^x Y^2$; so $b=0$.  Then
we must have $\dom(Y^5) + \dom(e^{-x} Y^2)= 0$, since otherwise
the sum $P(Y) \cto Y^5 + e^{-x} Y^2$.  So
$a^5+a^2 = 0$.  Since $a \ne 0$ and $a \in \R$,
we have $a=-1$.
[To consider also complex zeros of $P$, we would try
to use complex-valued transseries, and then the
other two cube roots of $-1$ would also need
to be considered here.]
Thus $Y = -e^{x/3}(1+S)$, where $S \fst 1$.  Then
\begin{align*}
	& P\big({-e^{x/3}} (1+S)\big) 
	\\ & =
	\big(-e^{x/3}(1+S)\big)^5 +
	e^x\big(-e^{x/3}(1+S)\big)^2 - x \big(-e^{x/3}(1+S)\big)
	- 9
	\\ & =
	-3e^{5x/3} S -9e^{5x/3}S^2 -10e^{5x/3}S^3
	\\ &\qquad -5e^{5x/3}S^4
	-e^{5x/3}S^5
	 +
	xe^{x/3}S+xe^{x/3}-9.
\end{align*}
Since $S$ is small, among the terms involving $S$ the dominant
one is $-3e^{5x/3}S$.  Solve $P\big({-e^{x/3}}(1+S)\big)=0$
for that term, and write the equation as $S = \Phi(S)$,
where
\begin{equation*}
	\Phi(S) := 
	-3S^2 -\frac{10}{3}S^3 -\frac{5}{3}S^4
	-\frac{1}{3}S^5+\frac{1}{3}xe^{-4x/3}S+\frac{1}{3}xe^{-4x/3}
	-3e^{-5x/3}.
\end{equation*}
Start with any $S_0$
and iterate $S_1 = \Phi(S_0)$, $S_2 = \Phi(S_1)$, etc.  For example,
{\allowdisplaybreaks
\begin{align*}
	S_0 &= 0
	\\
	S_1 &=  \frac{1}{3}xe^{-4x/3} - 3 e^{-5x/3}
	\\
	S_2 &=  \frac{1}{3}xe^{-4x/3} - 3 e^{-5x/3} -\frac{2}{9}x^2e^{-8x/3}
	+ 5 x e^{-3x} - 27 e^{-10x/3}
	- \frac{10}{81}x^3e^{-4x}
	\dots
	\\
	S_3 &=  \frac{1}{3}xe^{-4x/3} - 3 e^{-5x/3} -\frac{2}{9}x^2e^{-8x/3}
	+ 5 x e^{-3x} - 27 e^{-10x/3}
	+ \frac{20}{81}x^3e^{-4x}
	\dots
\end{align*}
}%
Each step produces more terms that subsequently remain unchanged.
Thus we get a solution for $P(Y)=0$ in the form
{\allowdisplaybreaks
\begin{align*}
	Y =& -e^{x/3}-\frac{1}{3}xe^{-x}+3e^{-4x/3}
	+\frac{2}{9}x^2e^{-7x/3}
	-5xe^{-8x/3}
	+27e^{-3x}
	\\&- \frac{20}{81}x^3e^{-11x/3}
	+9x^2e^{-4x}
	-105xe^{-13x/3}
	+396e^{-14x/3}
	+\frac{1}{3}x^4e^{-5x}
	\\&- \frac{455}{27}x^3e^{-16x/3}
	+308x^2e^{-17x/3}
	-2430xe^{-6x}
	-\frac{364}{729}x^5e^{-19x/3}
	\\&+ 7020e^{-19x/3}
	+\frac{2618}{81}x^4e^{-20x/3}
	-810x^3e^{-7x}
	+\o(e^{-7x}) .
\end{align*}
}%
The ``little o'' on the end represents, as usual, a remainder
that is ${}\fst e^{-7x}$.

Now consider the other possibility, $L=-x/2$.  Using the
same reasoning as before, we get $b=0$ and $a^2-9=0$,
so there are two possibilities $a = \pm 3$.  With the same
steps as before, we end up with two more solutions,
{\allowdisplaybreaks
\begin{align*}
	Y = {}& \pm3e^{-x/2}+\frac{1}{2}xe^{-x}
	\pm\frac{1}{24}x^2e^{-3x/2}
	\mp\frac{1}{3456}x^4e^{-5x/2}
	-\frac{81}{2}e^{-3x}
	\\& \pm\frac{1}{248832}x^6e^{-7x/2}
	\mp\frac{135}{4}xe^{-7x/2}
	-\frac{27}{2}x^2e^{-4x}
	\mp\frac{5}{71663616}x^8e^{-9x/2}
	\\& \mp\frac{105}{32}x^3e^{-9x/2}
	-\frac{1}{2}x^4e^{-5x}
	\pm\frac{7}{5159780352}x^{10}e^{-11x/2}
	\mp\frac{21}{512}x^5e^{-11x/2}
	\\& \pm\frac{19683}{8}e^{-11x/2}
	+3645xe^{-6x}
	\mp\frac{7}{247669456896}x^{12}e^{-13x/2}
	\\& \pm\frac{11}{36864}x^7e^{-13x/2}
	\pm\frac{168399}{64}x^2e^{-13x/2}
	+1215x^3e^{-7x}
	+\o(e^{-7x}) .
\end{align*}
}%
It turns out that these three transseries solutions converge
for large enough $x$.
As a check, take $x=10$ in $P$.  Maple says
the zeros are
\begin{align*}
	-28&.0317713673296286443879064009, 		
	\\-0&.0199881159048462608264265543923,
	\\0&.0204421151948799622524221088662,
	\\14&.0156586840197974714809554232+24.27610347738050183477184 i,
	\\14&.0156586840197974714809554232-24.27610347738050183477184 i .
\end{align*}
Plugging $x=10$ in the three series shown above (up to order
$e^{-7x}$), I get
\begin{align*}
	-28&.0317713673296286443879064\ 149,
	\\-0&.0199881159048462608264265\ 439647 ,
	\\0&.020442115194879962252422\ 0981049 .
\end{align*}

\subsection*{A Derivative and a Borel Summation}
Consider the Euler series
\begin{equation*}
	S = \sum_{k=0}^\infty \frac{k! e^x}{x^{k+1}}.
\end{equation*}
Differentiate term-by-term to get a telescoping sum
leaving only $S' = e^x/x$.  This means any summation method
that commutes with summation of series and differentiation
should yield the exponential integral function for $S$.
In fact, this is a case that can be done by classical Borel summation
(Example~\ref{eulerborel}).

\subsection*{A Compositional Inverse}
\begin{problm}\label{lambertseries}
Compute the compositional inverse of $x e^x$
\end{problm}
This inverse
is known as the Lambert $W$ function.  There is
a standard construction for all compositional inverses,
but we will proceed here directly.
First we need to know the dominant term.
This is done by ``reducing to height zero'' as follows.
If $x = W e^W$, then $\log x = W + \log W$,
so $W = \log x - \log W$ with $\log W \fst W$, and
thus $W \ato \log x$.

So assume our inverse is $\log x + Q$, with $Q \fst \log x$.  Then
$x = \big(\log x + Q\big) e^{\log x + Q}$ so
$x = (\log x + Q) x e^Q$ so
$e^{-Q} = \log x + Q$.  Now we should solve for one $Q$ in
terms of the other(s), and
use this to iterate.  If we take $Q = e^{-Q}-\log x$ and
iterate $\Phi(Q) = e^{-Q}-\log x$, it
doesn't work:  starting with $Q_0 = 0$, we get
$Q_1 = 1-\log x$, then $Q_2 = x/e - \log x$,
which is not converging.

So we will solve for the other one:
$Q = -\log(\log x + Q)$.  Write $\Phi(Q) = -\log(\log x + Q)$
and iterate.
Since we assume $Q \fst \log x$, the term
$Q/\log x$ is small.  So write
\begin{align*}
	\Phi(Q) &= 
	-\log\big(\log x + Q\big)
	= -\log\big((\log x) (1 + Q/\log x)\big)
	\\
	&=  -\log \log x + \sum_{k=1}^\infty \frac{(-1)^k}{k}
	\left(\frac{Q}{\log x}\right)^k .
\end{align*}
We will start with ratios $\mu_1 = 1/\log\log x$,
$\mu_2 = 1/\log x$.  So
\begin{equation*}
	\Phi(Q) = -\mu_1^{-1} + \sum_{k=1}^\infty \frac{(-1)^k}{k}
	(Q\mu_2)^k .
\end{equation*}
Start with $Q_0 = 0$.  Then $Q_1$ begins $-\mu_1^{-1}$,
so for the series in $\Phi(Q_1)$ we need $\mu_1^{-1} \mu_2 \fst 1$.
Of course $\mu_1^{-1} \mu_2 = \log\log x/\log x$ actually is small,
but not $\{\mu_1,\mu_2\}$-small.  So we add another ratio,
$\mu_3 = \log\log x/\log x = \mu_1^{-1}\mu_2$.  Now
computing with $\mu_2$ and $\mu_3$, iteration of
\begin{equation*}
	\Phi(Q) = -\mu_2^{-1}\mu_3 +
	\sum_{k=1}^\infty \frac{(-1)^k}{k}(Q\mu_2)^k
\end{equation*}
is just a matter of routine:
\begin{align*}
	Q_0 &=  0
	\\
	Q_1 &=  -\mu_2^{-1}\mu_3
	\\
	Q_2 &=  -\mu_2^{-1}\mu_3 + \mu_3 + \frac{1}{2}\mu_3^2 + \frac{1}{3}\mu_3^3
	+ \frac{1}{4}\mu_3^4 + \frac{1}{5}\mu_3^5 +\dots
	\\
	Q_3 &=  -\mu_2^{-1}\mu_3 + (1-\mu_2)\mu_3 +
	\left(\frac{1}{2}-\frac{3}{2}\mu_2+\frac{1}{2}\mu_2^2\right)\mu_3^2+\dots
\end{align*}
When we continue this, we get more and more terms which remain the
same from one step to the next.
I did this with Maple, keeping
terms with total degree at most $6$ in $\mu_2,\mu_3$.  When
it stops changing, I know I have the first terms of the answer.
Substituting in the values for $\mu_2$ and $\mu_3$,
writing $\l_1 = \log x$ and $\l_2 = \log\log x$, we get:
\begin{align*}
W(x) = {}& \l_1 -\l_2+\frac{\l_2}{\l_1}
+\frac{1}{2} \frac{\l_2^2}{\l_1^2}-\frac{\l_2}{\l_1^2}
+\frac{1}{3} \frac{\l_2^3}{\l_1^3}-\frac{3}{2} \frac{\l_2^2}{\l_1^3}
+\frac{\l_2}{\l_1^3}+\frac{1}{4} \frac{\l_2^4}{\l_1^4}
-\frac{11}{6} \frac{\l_2^3}{\l_1^4}+\frac{3 \l_2^2}{\l_1^4}
\\& -\frac{\l_2}{\l_1^4}+\frac{1}{5} \frac{\l_2^5}{\l_1^5}
-\frac{25}{12} \frac{\l_2^4}{\l_1^5}
+\frac{35}{6} \frac{\l_2^3}{\l_1^5}-\frac{5 \l_2^2}{\l_1^5}
+\frac{\l_2}{\l_1^5}+\frac{1}{6} \frac{\l_2^6}{\l_1^6}
-\frac{137}{60} \frac{\l_2^5}{\l_1^6}+\frac{75}{8} \frac{\l_2^4}{\l_1^6}
\\& -\frac{85}{6} \frac{\l_2^3}{\l_1^6}+\frac{15}{2} \frac{\l_2^2}{\l_1^6}
-\frac{\l_2}{\l_1^6}+\cdots .
\end{align*}

\section*{Contractive Mappings}
There is a general principle that explains why the sort of iterations
that we have seen will work.  It is a sort of ``fixed-point''
theorem for an appropriate type of ``contraction'' mappings.
Here is an explanation.

First consider a domination relation for sets
of multi-indices.

\begin{de}\label{def:dominate}
Let $\BE, \BF$ be subsets of $\Z^n$.  We say
$\BE$ \Def{dominates} $\BF$ iff for every $\bk \in \BF$, there is
$\bp \in \BE$ with $\bp < \bk$.
\end{de}

This may seem backward.  But correspondingly in the realm of transmonomials,
we will say larger monomials dominate smaller ones.

It's transitive: If $\BE_1$ dominates $\BE_2$ and $\BE_2$ dominates $\BE_3$,
then $\BE_1$ dominates $\BE_3$.  Every $\BE$ dominates $\emptyset$.

Recall (Proposition~\ref{magfin}) that $\Min \BE$
denotes the set of minimal elements of $\BE$.
And $\Min \BE$ is finite
if $\BE \subseteq \J_\bm$ for some $\bm$.

\begin{pr}\label{multistart}
Let $\BE, \BF$ be subsets of $\J_\bm$.  Then
$\BE$ dominates $\BF$ if and only if $\Min \BE$
dominates $\Min \BF$.
\end{pr}

\begin{proof}
Assume $\BE$ dominates $\BF$.  Let $\bk \in \Min \BF$.
Then $\bk \in \BF$, so there is $\bk_1 \in \BE$
with $\bk_1 < \bk$.  Then there is $\bk_0 \in \Min \BE$
with $\bk_0 \le \bk_1$.  So $\bk_0 < \bk$.

Conversely, assume $\Min \BE$ dominates $\Min \BF$.
Let $\bk \in \BF$.  Then there is $\bk_1 \in \Min \BF$
with $\bk_1 \le \bk$.  So there is $\bk_0 \in \Min \BE$
with $\bk_0 < \bk_1$.  Thus $\bk_0 \in \BE$ and
$\bk_0 < \bk$.
\end{proof}

\begin{pr}\label{multiproofs}
Let $\BE, \BF \subseteq \J_\bm$.
If $\BE$ dominates $\BF$, then $\Min \BE$ and $\Min \BF$ are disjoint.
\end{pr}
\begin{proof}
Assume $\BE$ dominates $\BF$.  If $\bk \in \Min \BF$, then $\bk \in \BF$,
so there is $\bk_1 \in \BE$ with $\bk_1 < \bk$.  So
even if $\bk \in \BE$, it is not minimal.
\end{proof}

\begin{pr}\label{Edominateprop}
Let $\BE_j \subseteq \J_\bm$, $j \in \N$, be an infinite sequence such that
$\BE_j$ dominates $\BE_{j+1}$ for all $j$.  Then the sequence
$(\BE_j)$ is point-finite; $\BE_j \to \emptyset$.
\end{pr}
\begin{proof}
Let $\bp \in \J_\bm$.  Then $\BF = \SET{\bk \in \J_\bm}{\bk \le \bp}$ is
finite.  But the sets $\BF \cap \Min \BE_j$ are disjoint
(by Proposition~\ref{multiproofs}), so all but finitely many
of them are empty.  For every $j$ with $\bp \in \BE_j$,
the set $\BF \cap \Min \BE_j$
is nonempty.  Therefore, $\bp \in \BE_j$ for only finitely
many $j$.
\end{proof}

\begin{pr}\label{pfinitedominate}
Let $\BE_i \subseteq \J_\bm$ be a point-finite family.
Assume $\BE_i$ dominates $\BF_i$ for all $i$.  Then
$(\BF_i)$ is also point-finite.
\end{pr}
\begin{proof}
Let $\bp \in \J_\bm$.  Then $\BF = \SET{\bk \in \J_\bm}{\bk < \bp}$ is
finite.  But the collection of sets $\BF \cap \Min \BE_j$
is point-finite, so again all but finitely many of them are nonempty.
For every $j$ with $\bp \in \BF_j$, the set
$\BF \cap \Min \BE_j$ is nonempty.  Therefore, $\bp \in \BF_j$ for only finitely
many $j$.
\end{proof}

Next consider the corresponding notion for a grid-based
field $\T^{\ebmu}$ of transseries.

\begin{de}
For $\m,\n \in \GRID^{\ebmu}$, we write $\m \fst^{\ebmu} \n$
and we say $\n$ \Def{$\bmu$-dominates} $\m$ iff
$\m/\n \fst^{\ebmu} 1$ (that is, $\m/\n = \bmu^\bk$ for
some $\bk > \0$).
\end{de}

The following are easy. (They follow from Propositions
\ref{wellpartially}--\ref{magfin} using
Proposition~\ref{finitetoone}).
The grid $\GRID^{\ebmu,\bm}$ is well-partially-ordered for 
(the converse of) $\fgt^{\ebmu}$.

\begin{pr}
If $\AA \subseteq \GRID^{\ebmu,\bm}$, $\AA \ne \emptyset$, then there is a
$\bmu$-maximal element: $\m \in \AA$ and $\g \fgt^{\ebmu} \m$
for no $\g \in \AA$.
\end{pr}

\begin{pr}\label{mu-sequence}
Let $\AA \subseteq \GRID^{\ebmu,\bm}$ be infinite.  Then there is
a sequence $\g_j \in \AA$, $j \in \N$, with
$\g_0 \fgt^{\ebmu} \g_1 \fgt^{\ebmu} \g_2  \fgt^{\ebmu} \cdots$.
\end{pr}

\begin{pr}
Let $\AA \subseteq \GRID^{\ebmu,\bm}$.  Then the set $\Max^{\ebmu} \AA$
of $\bmu$-maximal elements of $\AA$ is finite.  For every $\g \in \AA$
there is $\m \in \Max^{\ebmu} \AA$ with $\g \fste^{\ebmu} \m$.
\end{pr}

\begin{de}\label{Edominates}
Let $\AA, \BB \subseteq \G$.  We say
$\AA$ \Def{$\bmu$-dominates} $\BB$ (and write $\AA \fgt^{\ebmu} \BB$) iff for
all $\fb \in \BB$ there exists $\fa \in \AA$
such that $\fa \fgt^{\ebmu} \fb$.  Let $S,T \in \T^{\ebmu}$. We say
$S$ \Def{$\bmu$-dominates} $T$ (and write $S \fgt^{\ebmu} T$) iff
$\supp S$ $\bmu$-dominates $\supp T$.  Note that this agrees
with the previous definitions for $\fgt^{\ebmu}$ when
$S=1$ or when $S,T \in \GRID^{\ebmu}$.
\end{de}

\begin{re}
$S \fgt T$ if and only if there exists $\bmu$ such that
$S \fgt^{\ebmu} T$.
\end{re}

\begin{re}
Use of $\fgt^{\ebmu}$ requires caution (at least for non-monomials),
because it does not always have the properties of $\fgt$.
For example: $\bmu$-dominance is not preserved by multiplication.
That is: $A \fst^{\ebmu} B$ does not imply $AS \fst^{\ebmu} BS$.
For example, let $\ebmu = \{x^{-1},e^{-x}\}$,
$A = x^{-2}+e^{-2x}$, $B = x^{-1}+e^{-x}+xe^{-2x}$, and
$S = x^{-1}-e^{-x}$.  The term $x^{-1}e^{-2x}$ of $AS$
is not $\bmu$-dominated by any term of $BS$.
\end{re}

The following four propositions are proved as in multi-indices
(Propositions \ref{multistart} to~\ref{pfinitedominate}).

\begin{pr}
Let $\AA, \BB \subseteq \GRID^{\ebmu,\bm}$.  Then
$\AA \fgt^{\ebmu} \BB$ if and only if
$\Max^{\ebmu} \AA \fgt^{\ebmu} \Max^{\ebmu} \BB$.
\end{pr}

\begin{pr}
Let $\AA, \BB \subseteq \GRID^{\ebmu,\bm}$.
If $\AA \fgt^{\ebmu} \BB$, then $\Max^{\ebmu} \AA$ and $\Max^{\ebmu} \BB$
are disjoint.
\end{pr}

\begin{pr}\label{mudominateprop}
Let $\AA_j \subseteq \GRID^{\ebmu,\bm}$, $j \in \N$, be an infinite
sequence such that $\AA_j \fgt^{\ebmu} \AA_{j+1}$ for all $j$.
Then the sequence $(\AA_j)$ is point-finite
\end{pr}

\begin{pr}\label{mufinitedominate}
Let $\AA_i \subseteq \GRID^{\ebmu,\bm}$ be a point-finite family.
Assume $\AA_i \fgt^{\ebmu} \BB_i$ for all $i$.  Then
$\BB_i \subseteq \GRID^{\ebmu,\bm}$, $i \in \N$, and
the family $(\BB_i)$ is also point-finite.
\end{pr}

\begin{de}
Let $\Phi$ be linear from some subspace of $\T^{\ebmu}$ to itself.
Then we say $\Phi$ is \Def{$\bmu$-contractive} iff
$T \fgt^{\ebmu} \Phi(T)$ for all $T$ in the subspace.
\end{de}

\begin{de}
Let $\Phi$ be possibly non-linear from
some subset $\A$ of $\T^{\ebmu}$ to itself.
Then we say $\Phi$ is \Def{$\bmu$-contractive} iff
$(S-T) \fgt^{\ebmu} \big(\Phi(S)-\Phi(T)\big)$ for all $S,T \in \A$
with $S \ne T$.
\end{de}

There is an easy way to define a linear $\bmu$-contractive
map $\Phi$.
If $\Phi$ is defined on all monomials
$\g \in \AA \subseteq \GRID^{\ebmu,\bm}$
and $\g \fgt^{\ebmu} \Phi(\g)$ for them,
then the family $(\supp \Phi(\g))$ is point-finite by
Proposition~\ref{mufinitedominate}, so
\begin{equation*}
	\Phi\left(\sum c_\g \g\right) = \sum c_\g \Phi(\g)
\end{equation*}
$\bmu$-converges and defines $\Phi$ on the span.

\begin{ex}
The set $\bmu$ of ratios is important.  (In fact, this is the
reason we have been paying so much attention to the ratio
set $\bmu$.)  We cannot simply replace
``$\bmu$-small'' by ``small'' in the definitions.
Suppose $\Phi(x^{-j}) = x^j e^{-x}$ for all
$j \in \N$, and $\Phi(\g) = \g x^{-1}$ for all other monomials.
Then $\g \fgt \Phi(\g)$ for all $\g$.  But $\Phi(\sum x^{-j})$ evaluated
termwise is not a legal transseries.  Or:
Define $\Phi(x^{-j}) = e^{-x}$ for all
$j \in \N$, and $\Phi(\g) = \g x^{-1}$ for all other monomials.
Again $\g \fgt \Phi(\g)$ for all $\g$, but the family $\supp \Phi(x^{-j})$
is not point-finite.
\end{ex}

\begin{thm}[Grid-Based Fixed-Point Theorem]\label{costinfixed}
{\rm (i)} If $\Phi$ is linear and $\bmu$-con\-trac\-tive on $\T^{\ebmu,\bm}$,
then for any $T_0 \in \T^{\ebmu,\bm}$, the fixed-point equation
$T = \Phi(T)+T_0$ has a unique solution $T \in \T^{\ebmu,\bm}$.
{\rm (ii)}~If $\A \subseteq \T^{\ebmu,\bm}$ is nonempty and closed
{\rm (}in the asymptotic topology{\rm )}, and nonlinear
$\Phi \takes \A \to \A$ is $\bmu$-contractive on $\A$, then
$T=\Phi(T)$ has a unique solution in $\A$.
{\rm (From \cite[Theorem~15]{costintop}.  See
\cite[\S 6.5]{hoeven} and \cite{hoevenoperators}.)}
\end{thm}
\begin{proof}
(i) follows from (ii), since if $\Phi$ is linear and $\bmu$-contractive,
then $\widetilde{\Phi}$ defined by $\widetilde{\Phi}(T) = \Phi(T)+T_0$
is $\bmu$-contractive.

(ii) First note $\Phi$ is $\bmu$-continuous:  Assume
$T_j \muto T$.
Then $T_j - T \muto 0$, so $(\supp(T_j - T))$
is point-finite.  But $\supp(T_j - T) \fgt^{\ebmu} \supp(\Phi(T_j)-\Phi(T))$,
so $(\supp(\Phi(T_j)-\Phi(T))$ is also
point-finite by Proposition~\ref{mufinitedominate}.
And so $\Phi(T_j) \muto \Phi(T)$.

Existence:  Define $T_{j+1} = \Phi(T_j)$.  We claim $T_j$
is $\bmu$-convergent.  The sequence $\AA_j = \supp(T_j - T_{j+1})$
satisfies: $\AA_j \fgt^{\ebmu} \AA_{j+1}$ for all $j$,
so (Proposition~\ref{mudominateprop})
$(\AA_j)$ is point-finite, which
means $T_j-T_{j+1} \muto 0$
and therefore (by nonarchimedean Cauchy)
$T_j$ $\bmu$-converges.  Difference preserves
$\bmu$-limits, so the limit $T$ satisfies $\Phi(T)=T$.

Uniqueness: if $A$ and $B$ were two different
solutions, then $\Phi(A)-\Phi(B) = A - B$, which contradicts
$\bmu$-contractivity.
\end{proof}

\begin{re}
The $\bmu$-dominance relation may be used to explain two of the earlier
results that may have seemed poorly motivated at the time.

(a) To prove the existence of the derivative: When $\g'$ had
been defined for $\g \in \GRID^{\ebmu,\bm}$, we then showed
(Proposition~\ref{Ederivexist})
that the set $\SET{\g'}{\g \in \GRID^{\ebmu,\bm}}$
is point-finite.  We could first show: given $\bmu$ there
exists $\widetilde{\bmu}$ such that if $\m \fst^{\ebmu} \n$,
$\n \ne 1$,
then $\m' \fst^{\tilde{\ebmu}} \n'$.  Then:
Given any $\n \in \G$, we claim that the set
$\AA = \SET{\g \in \GRID^{\ebmu,\bm}}{\n \in \supp \g'}$
is finite.  If not, by Proposition~\ref{mu-sequence} there is an infinite
sequence $\g_j \in \AA$, $\g_j \ne 1$, with $\g_0 \fgt^{\ebmu} \g_1 \fgt^{\ebmu} \cdots$.
But then $\g'_0 \fgt^{\tilde{\ebmu}} \g'_1 \fgt^{\tilde{\ebmu}} \cdots$,
so $\SET{\supp \g'_j}{j \in \N}$ is point-finite by
Proposition~\ref{mudominateprop}, contradicting
the assumption that $\AA$ is infinite.

(b) To prove the existence of the composition $T \circ S$: When
$\g \circ S$ had been defined for $\g \in \GRID^{\ebmu,\bm}$, we then
showed (Proposition~\ref{Ecompexist})
that the set $\SET{\g\circ S}{\g \in \GRID^{\ebmu,\bm}}$
is point-finite.  We could first show: given $\bmu$ and $S$,
there exists $\widetilde{\bmu}$ such that if $\m \fgt^{\ebmu} \n$,
then $\m \circ S \fgt^{\tilde{\ebmu}} \n \circ S$.  Then:
Given any $\n \in \G$, we claim that the set
$\AA = \SET{\g \in \GRID^{\ebmu,\bm}}{\n \in \supp(\g \circ S)}$
is finite.  If not, by Proposition~\ref{mu-sequence} there is an infinite
sequence $\g_j \in \AA$ with $\g_0 \fgt^{\ebmu} \g_1 \fgt^{\ebmu} \cdots$.
But then $\g_0\circ S \fgt^{\tilde{\ebmu}} \g_1\circ S
\fgt^{\tilde{\ebmu}} \cdots$,
so $\SET{\supp(\g_j\circ S)}{j \in \N}$ is point-finite by
Proposition~\ref{mudominateprop}, contradicting
the assumption that $\AA$ is infinite.
\end{re}

\section*{Integration}
In elementary calculus courses, we find that certain integrals
can be evaluated using reduction formulas.  For example
$\int x^n e^x\,dx$, when integrated by parts, yelds an integral
of the same form, but with exponent $n$ reduced by $1$.
So if we repeat this until the exponent is zero, we have
our integral.  But of course this does not work
when the exponent is not in $\N$.  We can try it, and
get an infinite series:
{\allowdisplaybreaks
\begin{align*}
	\int x^a e^x\,dx &= 
	x^a e^x - a \int x^{a-1} e^x\,dx
	\\&= x^a e^x - a x^{n-1} e^x + a(a-1)\int x^{a-2}e^x\,dx
	= \cdots
	\\&= 
	\sum_{j=0}^\infty (-1)^ja(a-1)(a-2)\cdots(a-j+1)\; x^{a-j}e^x
	\\&= 
	\sum_{j=0}^\infty \frac{\Gamma(j-a)}{\Gamma(-a)}\,x^{a-j} e^x ,
\end{align*}
}%
but this series converges for no $x$. However, it is still
a transseries solution to the problem.

\begin{pr}\label{int_xaebec}
Let $a,b,c \in \R$, $c>0, b\ne 0$.  Then the transseries
\begin{equation*}
	T = \sum_{j=1}^\infty
	\frac{\Gamma\left(j-\frac{a+1}{c}\right)}{\Gamma\left(1-
	\frac{a+1}{c}\right)cb^j}\,
	x^{a+1-jc} e^{bx^c}
\end{equation*}
has derivative $T' = x^a e^{bx^c}$.  {\rm (}If $(a+1)/c$ is
a positive integer, then $T$ should be a finite sum.{\rm )}
\end{pr}

\begin{problm}
More generally: if $b \in \R$
and $L \in \T_0$ is purely large,
can you use the same method to show that there is
$T \in \T_1$ with $T' = x^be^L$?
\end{problm}

\subsection*{The General Integral}
Every transseries in $\R\lbbb x \rbbb$
has an integral (an antiderivative).
We will give a complete proof following the
hint in \cite[4.10e.1]{costinasymptotics}.
This is an example where we convert the problem
to a log-free case to apply the contraction argument.
The general integration problem (Theorem~\ref{Eintegral}) is
reduced to one (Proposition~\ref{intlarge}) where contraction
can be easily applied.  

\begin{pr}\label{intlarge}
Let $T \in \T_{\bullet}$ with $T \fgt 1$.  Then
there is $S \in \T_{\bullet}$ with $S' = e^T$.
\end{pr}
\begin{proof}
Either $T$ is positive or negative.  We will do the positive case,
the negative one is similar (and it turns out the iterative
formulas are the same).  If
\begin{equation*}
	S = \frac{e^T}{T'}\,(1 + U),
\end{equation*}
where $U$ satisfies
\begin{equation*}
	U = \frac{T''}{(T')^2} +\frac{T''}{(T')^2}U
	-\frac{U'}{T'} ,
\end{equation*}
then it is a computation to see that $S' = e^T$.
So it suffices to exhibit an appropriate $\bmu$ and
show that the linear map $\Phi \takes \T^{\ebmu,\0} \to \T^{\ebmu,\0}$
defined by
\begin{equation*}
	\Phi(U) = \frac{T''}{(T')^2}U - \frac{U'}{T'}
\end{equation*}
is $\bmu$-contractive, then apply Theorem~\ref{costinfixed}(i).

Say $T$ is of exact height $N$, so $e^T$ is of exact
height $N+1$.  By Proposition~\ref{EWKB}, $T'' \fst (T')^2$ and $xT' \fgt 1$.
So $T''/(T')^2$ and $1/(xT')$ are small.  Let
$\tbmu$ be the least set of ratios
including $x^{-1}$, ratios generating
$T$, the inversion addendum for $T'$, the smallness addenda
for $T''/(T')^2$ and $1/(xT')$, and is hereditary.
Then, for each $\mu_i = x^{-b_i} e^{-L_i}$ in $\tbmu$
(finitely many of them), since $L_i$ has lower height than
$T'$, we have $L'_i/T' \fst 1$.  Add smallness addenda
for all of these, call the result $\bmu$.
Note $\lsupp \bmu = \lsupp \tbmu$, so we
don't have to repeat this last step.

By Proposition~\ref{addendumheight} all ratios
in $\bmu$ are (at most) of height $N$.
And all derivatives $T', T''$ belong to $\T^{\ebmu}$.
The function $\Phi$ maps $\T^{\ebmu}$ into itself.

Since $\Phi$ is linear, we just have to check for
monomials $\g \in \GRID^{\ebmu,\0}$ that
$\g \fgt^{\ebmu} \Phi(\g)$.
Now $T''/(T')^2$ is $\bmu$-small
so $\g \fgt^{\ebmu} (T''/(T')^2) \g$.  For the
second term:   If $\g = \bmu^\bk = x^b e^L$,
then
\begin{equation*}
	\frac{\g'}{T'} = \frac{bx^{b-1}e^L+L'x^b e^L}{T'} =
	\frac{bx^{-1}+L'}{T'}\,\g =
	\frac{b}{xT'}\,\g + \frac{L'}{T'}\,\g.
\end{equation*}
But $1/xT' \fst^{\ebmu} 1$ so $\g \fgt^{\ebmu} (b/(xT'))\g$.
And $L'/T' \fst^{\ebmu} 1$ so
$\g \fgt^{\ebmu} (L'/T')\g$.
\end{proof}

\begin{de}
We say $x^be^L \in \G_{\bullet}$ is \Def{power-free} iff $b=0$.
We say $T \in \T_{\bullet}$ is power-free iff all transmonomials
in $\supp T$ are power-free.
\end{de}

Since $(x^b e^L) \circ \exp = e^{bx} e^{L \circ \exp}$,
it follows that all $T \in \T_{\bullet,-1}$ are power-free.

\begin{pr}\label{intsum}
Let $T \in \T_{\bullet}$ be a power-free transseries.  Then there is
$S \in \T_{\bullet}$ with $S'=T$.
\end{pr}
\begin{proof}
For monomials $\g=e^L \in \supp T$ with large $L \in \T^{\ebmu}$,
write $\P(\g)$ for the
transseries constructed
in Proposition~\ref{intlarge} with $\P(\g)'=\g$.  Then we
must show that the family $(\supp \P(\g))$ is point-finite,
so we can define $\P\big(\sum c_\g \g\big) = \sum c_\g \P(\g)$.
For large $L$ we have $x L' \fgt 1$ (Proposition~\ref{EWKB}).
Thus, the formula
\begin{equation*}
	\frac{\P(e^L)}{x} = \frac{e^L}{xL'}(1+U)
\end{equation*}
shows that $e^L$ $\tbmu$-dominates $\P(e^L)/x$ for
an appropriate $\tbmu$.  So the family
$\supp (\P(e^L)/x)$ is point-finite and thus
the family $\supp \P(e^L)$ is point-finite.
\end{proof}

\begin{thm}\label{Eintegral}
Let $A \in \R\lbbb x \rbbb$.  Then there exists $B \in \R\lbbb x \rbbb$
with $B'=A$.
\end{thm}
\begin{proof}
Say $A \in \T_{\bullet M}$.  Then $A = T_1 \circ \log_{M+1}$, where
$T_1 \in \T_{\bullet,-1}$.  Let
\begin{equation*}
	T = T_1 \cdot \exp_{M+1} \cdot \exp_M \cdots \exp_2 \cdot \exp_1 .
\end{equation*}
Now $T$ is power-free,
so by Proposition~\ref{intsum}, there exists $S \in \T_{\bullet}$
with $S' = T$.  Then let $B = S \circ \log_{M+1}$
and check that $B' = A$.  Note that $B \in \T_{\bullet,M+1}$.
\end{proof}

\subsection*{An Integral}
\begin{problm}\label{compint}
Compute the integral
\begin{equation*}
	\int e^{\displaystyle{e^{e^x}}}dx
\end{equation*}
using the method of Proposition~\ref{intlarge}.
\end{problm}

We first display the ratio set to be used, and derivatives:
\begin{align*}
&\mu_1 = x^{-1}, & &\mu_1' = -\mu_1^2, & &L_1 = \mu_1^{-1} = x,
\\
&\mu_2 = e^{-x} = e^{-L_1}, & &\mu_2' = -\mu_2, & & L_2 = \mu_2^{-1} = e^x,
\\
&\mu_3 = e^{-e^x} = e^{-L_2}, & &\mu_3' = -\mu_2^{-1}\mu_3,
& &L_3 = \mu_3^{-1} = e^{e^x},
\\
&\mu_4 = e^{-e^{e^x}} = e^{-L_3}, & &\mu_4' = -\mu_2^{-1}\mu_3^{-1}\mu_4,
& &L_4 = \mu_4^{-1} = e^{e^{e^x}}.
\end{align*}
The integral should have the form $(L_4/L_3')(1+U)$, where
$U$ satisfies
\begin{equation*}
	U = \frac{L_3''}{(L_3')^2} + \frac{L_3''}{(L_3')^2}U
	-\frac{U'}{L_3'}
	=
	(\mu_3 + \mu_2\mu_3) + (\mu_3 + \mu_2\mu_3)U - \mu_2\mu_3 U' .
\end{equation*}
To solve this, we should iterate
$U_{n+1} = (\mu_3 + \mu_2\mu_3) + \Phi(U_n)$ where
$\Phi(Y) = (\mu_3 + \mu_2\mu_3)Y - \mu_2\mu_3 Y'$.  Starting with
$U_0 = 0$, we get
\begin{align*}
	U_1 &= 
	(1+\mu_2)\mu_3,
	\\
	U_2 &= 
	(1+\mu_2)\mu_3+(2+3\mu_2+2\mu_2^2)\mu_3^2,
	\\
	U_3 &= 
	(1+\mu_2)\mu_3+(2+3\mu_2+2\mu_2^2)\mu_3^2+
	(6 + 11\mu_2 + 12\mu_2^2+6\mu_2^3)\mu_3^3,
\end{align*}
each step producing one higher power of $\mu_3$ and
preserving all of the existing terms.  Once we have
the limit $U$,
we add $1$ and multiply by $L_4/L_3' = e^{e^{e^x}}/(e^xe^{e^x})$.
The result is
\begin{equation*}
	\int e^{e^{e^x}}\,dx = e^{e^{e^x}}
	\sum_{j=1}^\infty e^{-j e^x}\left(\sum_{k=1}^j e^{-kx} c_{j,k}\right) .
\end{equation*}
The coefficients $c_{j,k}$ (namely, $1; 1,1; 2,3,2; 6,11,12,6;\cdots$)
are related to Stirling numbers of the
second kind.

Similarly, we may compute
\begin{equation*}
	\int e^{k_2 x} e^{k_3 e^x}e^{k_4 e^{e^x}}\,dx =
	e^{k_2 x} e^{k_3 e^x}e^{k_4 e^{e^x}}
	\sum_{j=1}^\infty e^{-j e^x}\left(\sum_{k=1}^j e^{-kx} c_{j,k}\right) .
\end{equation*}
for some coefficients $c_{j,k}$ depending on $k_2, k_3, k_4$.

\subsection*{A Differential Equation}
\begin{problm}
Solve the Riccati equation
\begin{equation*}
	Y' = \frac{x-x^2}{x^2-x+1}\,Y + Y^2 .
\tag*{($*$)}
\end{equation*}
\end{problm}

This is a differential equation where the solution
can be written in closed form.  (At least if you consider
an integral to be closed form.)  But it
will illustrate some things to watch out for when computing
transseries solutions.  The same things can happen in
cases where solutions are not known in closed form.

If we are not careful, we may come up with a series
\begin{equation*}
	S(x) = e^{-x}\left(1-\frac{1}{x}+\frac{1}{3x^3}
	+\frac{1}{6x^4}-\frac{1}{10x^5}-\frac{8}{45x^6}
	-\frac{1}{18x^7}+\frac{11}{120x^8}
	+\cdots\right)
\end{equation*}
and claim it is a solution.  If we plug this
series in for $Y$, then the two sides of the differential
equation agree to all orders.  That is, if we
compute $S$ up to $\O(e^{-x}x^{-1000})$, and plug it in,
then the two sides agree up to $\O(e^{-x}x^{-1000})$.
But in fact, $S$ is not the transseries solution of ($*$).
The two sides are not equal---their difference is just far smaller
than all terms of the series $S$.  The difference
has order $e^{-2x}$.  In hindsight, this should be clear,
because of the $Y^2$ term in~($*$).  If
$Y$ has a term $e^{-x}$ in its expansion, then
$Y^2$ will have a term $e^{-2x}$.  When $S$
is substituted into ($*$), the term $e^{-2x}$
appears on the right side but not the left.

In fact, $S(x)$ is a solution of ($*$) without
the $Y^2$ term.

According to Maple, the actual solution is
$Y = c S(x)/\big(1-c \int S(x)\,dx\big)$, where
$c$ is an arbitrary constant and
\begin{equation*}
	S(x) = \exp\left[-x+\frac{2}{\sqrt{3}}\arctan
	\left(\frac{2x-1}{\sqrt{3}}\right)\right] .
\end{equation*}
The exponent in this $S(x)$ is $-x$ plus constant plus small,
so $S$ can be written as a series.
It is (except for the constant factor) the
series $S(x)$ given above.

Now the integral of $S$ can be done (using Proposition~\ref{int_xaebec}),
then division carried out as usual.  The general solution of ($*$) is:
{\allowdisplaybreaks
\begin{align*}
	& c e^{-x} \left(1-\frac{1}{x}+\frac{1}{3x^3}+\frac{1}{6x^4}
	-\frac{1}{10x^5}-\frac{8}{45x^6}+\cdots\right)
	\\ +\; & c^2 e^{-2x} \left(-1+\frac{2}{x}
	-\frac{-2}{x^2}+\frac{7}{3x^3}-\frac{20}{3x^4}
	+\frac{388}{15x^5}-\frac{5578}{45x^6}+\cdots\right)
	\\ +\; & c^3 e^{-3x} \left(1-\frac{3}{x}+\frac{5}{x^2}
	-\frac{8}{x^3}+\frac{39}{2x^4}-\frac{693}{10x^5}
	+\frac{3159}{10x^6}+\cdots\right)
	\\ +\; & c^4 e^{-4x} \left(-1+\frac{4}{x}-\frac{9}{x^2}
	+\frac{53}{3x^3}-\frac{128}{3x^4}+\frac{707}{5x^5}
	-\frac{27442}{45x^6}+\cdots\right)
	\\ +\; & c^5 e^{-5x} \left(1-\frac{5}{x}+\frac{14}{x^2}
	-\frac{97}{3x^3}+\frac{487}{6x^4}-\frac{1549}{6x^5}
	+\frac{9509}{9x^6}+\cdots\right)
	\\ +\; & c^6 e^{-6x} \left(-1+\frac{6}{x}
	-\frac{20}{x^2}+\frac{53}{x^3}-\frac{141}{x^4}+\frac{2208}{5x^5}
	-\frac{8648}{5x^6}+\cdots\right)
	\\ +\; & \cdots
\end{align*}
}%
Transseries solutions to simple problems can have
support of transfinite order type!

The transseries solution to differential equation ($*$) can be found
without using a known closed form.  The generic method would
reduce to height zero (by taking logarithms of the
unknown $Y$) then solve as a contractive map.

There is another comment on doing these computations
with a computer algebra system.  Carrying out the
division indicated above, for example, is not
trivial.  If I write the two series to many terms,
divide, then tell Maple to write it as a series
(using the MultiSeries package, \texttt{series(A/B,x=infinity,15)}),
I get only the first row of the result above.  Admitedly,
there is a big-O term at the end, and all
terms in the subsequent rows are far smaller than
that, but it is not what we want here.

We want to discard not terms that are merely small, but
terms that are $\bmu$-small for a relevant $\bmu$.
So this computation can better be done using a grid.
Choose a finite ratio set---in this case
I used $\mu_1 = x^{-1}$, $\mu_2 = e^{-x}$.
We write the two series in terms of these ratios, then
expand the quotient as a series in the
two variables $\mu_1, \mu_2$.
Now we can control which terms
are kept.  Delete terms not merely when
they are small, but when they are $\bmu$-small.
The series above has all terms $\bmu^\bk$ with
$\bk \le (6,6)$.

I used this same grid method for the computations
in Problem~\ref{compint}.  That is the reason
I started there by displaying the required ratio set
and derivatives.

\subsection*{Factoring}
\begin{problm}
Factor the differential operator
$$
	\partial^2 + x\partial + I =
	\big(\partial - \alpha(x)I\big) \big(\partial - \beta(x)I\big) ,
$$
where $\alpha(x)$ and $\beta(x)$ are transseries.
\end{problm}

Why don't you do it?  My answer looks like this:
\begin{align*}
	\alpha(x) &=  -x +\frac{1}{x}+\frac{2}{x^3}+\frac{10}{x^5}
	+\frac{74}{x^7}+\frac{706}{x^9}+\frac{8162}{x^{11}}
	+\frac{110410}{x^{13}}+\O(x^{-15}),
	\\
	\beta(x) &=  -\frac{1}{x}-\frac{2}{x^3}-\frac{10}{x^5}
	-\frac{74}{x^7}-\frac{706}{x^9}-\frac{8162}{x^{11}}
	-\frac{110410}{x^{13}}+\O(x^{-15}).
\end{align*}
Are the coefficients Sloane A000698 \cite{sloane}?
I am told that these series are divergent, and that can be
proved by considering ``Stokes directions'' in
the complex plane---another interesting
topic beyond the scope of this paper.  Elementary functions
have convergent transseries \cite[Cor.~5.5]{DMM1}, so
$\alpha(x)$ and $\beta(x)$ (even the genuine functions
obtained by \'Ecalle--Borel summation)
are \emph{not} elementary functions.

\subsection*{Increasing and Decreasing}
I have (as part of the sales pitch) tried to show that the reasoning
required for the theory of transseries is easy, although
perhaps sometimes tedious.  But, in fact, I think there
are situations---dealing with composition---that are not as easy.

\begin{problm}
Let $T, A, B \in \T$.  Assume $A, B$ are large and positive.
Prove or disprove: if $T' > 0$ and $A < B$, then $T\circ A < T\circ B$.
\end{problm}

\section{Additional Remarks}\label{addremarks}
If (as I claim) the system $\R\lbbb x \rbbb$ of
transseries is an elementary and fundamental
object, then perhaps it is only natural that there are
variants in the formulation and definitions used.
For example \cite{hoeven} the construction can proceed by
first adding logarithms, and then adding exponentials.
For an exercise, see if you can carry that out
yourself in such a way that the end result is the same
system of transseries as constructed above.
I prefer the approach shown here, since I view the
``log-free'' calculations as fundamental.

There is a possibility \cite{asch, edgarc, hoeventhesis, kuhlmann, schmeling}
to allow \emph{well ordered supports}
instead of just the grids $\GRID^{\ebmu,\bm}$.
These are called \Def{well-based transseries}.
(Perhaps we use the alternate notation $\R[[\MM]]$ for the well-based
Hahn field and the new notation $\R\lbb\MM\rbb$ for
the grid-based subfield.)
The set of well-based transseries forms a strictly larger system
than the grid-based transseries, but with most of the same properties.
Which of these two is to be preferred may be
still open to debate.  In this paper we have
used the grid-based approach because:
\begin{enumerate}
\item[(i)] The finite ratio set is conducive to computer
calculations.
\item[(ii)] Problems from analysis almost always
have solutions in this smaller system.
\item[(iii)] Some proofs and formulations of definitions
are simpler in one system than in the other.
\item[(iv)] Perhaps (?) the analysis used for \'Ecalle--Borel
convergence can be applied only to grid-based series.
\item[(v)] \cite{shelah} In the well-based case, the domain of $\exp$
cannot be all of $\R[[\MM]]$.
\item[(vi)] \cite{harr} The grid-based ordered set $\R\lbb \MM \rbb$ is a
``Borel order,'' but the well-based ordered set $\R[[\MM]]$ is not.
\end{enumerate}

What constitutes desirable
properties of ``fields of transseries'' has been explored
axiomatically by Schmeling \cite{schmeling}.  Kuhlmann and Tressl
\cite{kuhlmanntressl} compare different constructions for
ordered fields of generalized power series with logarithm and exponential;
the grid-based transseries we
have used here might be thought of as
the smallest such system (without restricting the coefficients).

Just as the real number system $\R$ is extended to the
complex numbers $\C$, there are ways to extend the system
of real transseries to allow for complex numbers.
The simplest uses the same
group $\G_{\bullet\bullet}$ of monomials, but then takes
complex coefficients to form $\C\lbb \G_{\bullet\bullet} \rbb
= \C\lbbb x \rbbb$.
For example, the fifth-degree equation
in Problem~\ref{fifth} has five solutions
in $\C\lbb \G_{\bullet\bullet} \rbb$.  But this still won't
give us oscillatory functions, such as solutions
to the differential equation $Y''+Y=0$.  There is a
way \cite[Section~7.7]{hoeven} to define \Def{oscillating transseries}.
These are finite sums
\begin{equation*}
	\sum_{j=1}^n \alpha_j e^{i\psi_j},
\end{equation*}
with amplitudes $\alpha_j \in \C\lbbb x \rbbb$ and
purely large phases $\psi_j \in \R\lbbb x \rbbb$.
And van der Hoeven \cite{hoevencomplex} considers
defining complex transseries using the same method as
we used for real transseries, where the required
orderings are done
in terms of sectors in the complex plane.

\section*{Acknowledgements}
\addcontentsline{toc}{section}{Acknowledgements}%
Big thanks are owed to my colleague Ovidiu Costin, without whom
there would be no paper.  Useful comments were also provided
by Joris van der Hoeven, Jacques Carette, Chris Miller,
Jan Mycelski, Salma Kuhlmann, Bill Dubuque,
and an anonymous referee.

\end{document}